\documentclass[12pt]{amsart}

\usepackage{ulem}
\usepackage[english]{babel}
\usepackage{amssymb,enumerate,amsmath}
\usepackage{tikz}
\usepackage{array}

\usepackage{ifthen}
\usepackage[latin1]{inputenc}
\usepackage{amsmath,amssymb,amsthm}
\usepackage{color} 
\usepackage{graphicx}

\usepackage{cancel}

\newcommand{\Pl}{Pl\"ucker}
\newcommand{\rk}{\mathrm{rk}\,}

\newcommand{\hilbSchemeFunctor}[2]{\underline{\mathbf{Hilb}}^{#2}_{#1}}

\newcommand{\hilbScheme}[2]{\mathbf{Hilb}^{#2}_{#1}}

\newcommand{\PP}{{\mathbb{P}}}

\newcommand{\A}{{\mathbb{A}}}

\newcommand{\Proj}{\textnormal{Proj}}
\newcommand{\Hom}{\textnormal{Hom}}

\newcommand{\Imm}{\mathrm{Im}}
\newcommand{\Ker}{\mathrm{Ker}}
\newcommand{\Dim}{\mathrm{dim}}
\newcommand{\Codim}{\mathrm{codim}}

\newcommand{\GrassScheme}[2]{\mathbf{Gr}_{#1}^{#2}}

\newcommand{\fd}{\rightarrow}

 \numberwithin{equation}{section}
\newtheorem{lemma}{Lemma}[section] 
\newtheorem{theorem}[lemma]{Theorem} 
 
\newtheorem{proposition}[lemma]{Proposition} 
\newtheorem*{theoremA}{Theorem A}
\newtheorem*{theoremB}{Theorem B}

\theoremstyle{definition}
\newtheorem{definition}[lemma]{Definition} 
\newtheorem{notation}[lemma]{Notation} 
\newtheorem{remark}[lemma]{Remark} 
\newtheorem{example}[lemma]{Example}

\def\bcR{\color{red}}
\def\bcr{\color{red}}
\def\ecr{\color{black}}

\def\bco{\color{orange}}

\DeclareMathAlphabet{\mathpzc}{OT1}{pzc}{m}{it} 

\newenvironment{listecompacte}
{\begin{list}
    {\ensuremath{\bullet}}
    {\setlength{\topsep}{2pt}
      \setlength{\itemsep}{1pt} \setlength{\parsep}{0pt}}
}
{\end{list}
}

\begin{document}
\sloppy
\title{Quivers and equations a la Pl\"ucker for the Hilbert scheme}
\date{}
\author{Laurent Evain and Margherita Roggero}
\maketitle

\newcommand{\rg}{r} 

\newcommand{\mhu}[1]{\mu_{{#1}}} 
\newcommand{\M}[1]{M_{#1}} 
\newcommand{\matm}{M}
\newcommand{\Q}{Q} 
\newcommand{\n}{n} 
\newcommand{\ro}{\rho} 
\newcommand{\dg}{d} 
\newcommand{\p}{p} 
\newcommand{\q}{q} 

\newcommand{\algb}{B} 

\newcommand{\spec}[1]{\mathrm{Spec}({#1})} %
\newcommand{\gl}[1]{GL_{{#1}}} %
\newcommand{\mon}[1]{{\bf x}^{#1}} 
\newcommand{\pk}[1]{P_{{#1}}} 
\newcommand{\C}{C} 

\newcommand{\h}[1]{\fxh_{{#1}}} 
\newcommand{\localg}[0]{\fxg} 

\newcommand{\kk}{k} %
\newcommand{\kkbar}{{\overline k}} %

\newcommand{\marginnote}[1]{\marginpar{\footnotesize{#1}}}
\newcommand{\lt}[1]{\marginpar{ \color{blue} \footnotesize{#1}}}
\newcommand{\mgh}[1]{\marginpar{ \color{red} \footnotesize{#1}}}




\section*{Abstract: }
Several moduli spaces parametrising linear subspaces 
of the projective space are cut out by linear and quadratic equations
in their natural embedding: Grassmannians, 
Flag varieties, and Schubert varieties. 
The goal of this paper is to prove that a similar statement 
holds when one replaces linear subspaces with
algebraic subschemes of the projective space.
We exhibit equations of degree 1 and 2 that 
define schematically the Hilbert scheme 
 $\hilbScheme{\PP^n}{\p}$ 
 (for all,  possibly non-constant,
Hilbert polynomial $\p$) 
in its standard embedding $\hilbScheme{\PP^n}\p\hookrightarrow \GrassScheme{S_{R+1}}{\p(R+1)}$ with $R$ 
 any degree     larger than or equal to the Gotzmann number $r$ of
 $p$. For every $R< r$ these linear and quadratic  equations
 constructed,  and suitable linear inequalities define the locally
 closed embedding $\hilbScheme{\PP^n}{p, [R]}\hookrightarrow
 \GrassScheme{S_{R+1}}{\p(R+1)}$ of the Hilbert scheme parametrising
 subschemes with regularity upper bounded by $R$. 

The equations 
are  reminiscent of the Pl\"ucker relations for
Grassmannians: they are explicit and built formally with permutations 
on indexes on the Pl\"ucker coordinates. 
Our method relies on a new  description   of the Hilbert scheme 
as a quotient of a scheme of quiver representations.

\section{Introduction}
\label{sec:introduction}

\bcR

\bigskip

\ecr

The Pl\"ucker coordinates on a Grassmannian satisfy the well known 
Pl\"ucker relations \cite{KleimanLaksov}. Similarly, the flag varieties are defined 
by quadratic equations and Schubert varieties are defined by quadratic
and linear equations \cite{ramanathan87:equationsForSchubertVarieties, fultonYoungTableaux}.
The goal of this paper is to prove that, 
in a similar way, the Hilbert schemes parametrising 
closed subschemes of a projective space are defined by simple explicit
linear and quadratic equations in their natural embeddings. 

The Hilbert schemes 
carry in general a natural non reduced structure inherited from their functorial
construction. Our equations take into account the non reduced structure and
define the Hilbert schemes schematically.

Let $\hilbScheme{\PP^n}p$ be the Hilbert scheme
parametrising closed subschemes  of $\PP^n$   with
Hilbert polynomial $\p$. It can be embedded in the  Grassmannian  $\GrassScheme{S_R}{p(R)}$, 
where $R$ is any integer 
larger or equal to the Castelnuovo-Mumford-Gotzmann number $r$ of $p$ and $S_{R}=H^0 \mathcal O_{\PP^n}(R)$. 
Composing with
the Pl\"ucker embedding $\GrassScheme{S_{R}}{\p(R)}\subset \PP^{D(R)}$, 
 $D(R):={\binom{\Dim S_{R}}{\p(R)}}-1$, we consider the problem of finding equations 
for the Hilbert scheme in $\PP^{D(R)}$.

The question of finding equations for the Hilbert scheme as a subscheme of  a Grassmannian has been addressed many times after 
its  introduction  by Grothendieck.

The equations that arise depend much on the
way the Hilbert scheme is constructed.  The initial
construction by Grothendieck involved flattening stratifications
\cite[Lemme 3.4]{grothendieck60:techniquesConstructionEtSchemasDeHilbert}.  
Techniques were
developed to compute local equations for the flat stratum corresponding 
to the Hilbert scheme \cite{galligo}\cite[Proposition 0.5]{granger}.

 The work by
Gotzmann \cite{gotzmann} leads to a description of the Hilbert scheme 
as a determinantal locus in a product of Grassmannians. 
A new description  for the Hilbert scheme  as a subscheme of a single Grassmannian given    by local determinantal conditions  was
conjectured by Bayer in his PhD thesis \cite{Bayer82} and proved by Iarrobino and Kleiman in
 \cite[Appendix C]{iarrobino-kleiman:around_gotzmann_number} also
 exploiting an argument of Grothendieck.  Haiman and Sturmfels obtain
  Bayer's description as a special case of  their own construction of the multigraded Hilbert
scheme \cite{haiman_sturmfels02:multigradedHilbertSchemes}. 
In \cite{BLMR}  and \cite{lella-roggero:functoriality}, Brachat, Lella, Mourrain and Roggero define the Hilbert 
scheme using   functors which involve  symmetries  of the Hilbert
scheme given by the action of \noindent $GL_n$.
Commuting matrices of multiplication by variables  and border bases  have been applied to define equations for  Hilbert
schemes of points by Alonso, Brachat and Mourrain  \cite{ABM}.

The various approaches 
lead to equations of different degrees:  degree $n+1$, only
depending on the \lq\lq ambient\rq\rq\ space $\PP^n$, for those by
 Bayer, Iarrobino-Kleiman and Haiman-Sturmfels,     degree $deg(p)+2$, only depending
on the Hilbert polynomial, for those by
Brachat-Lella-Mourrain-Roggero.  

We will see that it is possible to 
find equations of degree 1 and 2 that cut out the
Hilbert scheme for every, possibly nonconstant $p$. These are obviously the smallest 
possible degrees since  in general  Hilbert schemes are not   linear spaces, not
even   linear sections of a Grassmannian \cite[Section 7.2]{BLMR}.

It was remarked by Haiman and Sturmfels
\cite{haiman_sturmfels02:multigradedHilbertSchemes}
that the framework of 
a quite theoretical construction of the Hilbert scheme
provides access to equations hardly accessible by direct computation. 
In cryptography, systems built with rich structures are possibly fragile
because attackers may extract information from the structure.
The above list of examples suggest that similarly
each  new   description  of the Hilbert scheme could expose a structure 
providing access to some new sets of equations. 

Starting from these remarks  the natural question is: how to 
produce a new  description  for the Hilbert scheme?

We considered the construction by Nakajima of $\hilbScheme{\A^2}{p}$, 
 for constant $p$   \cite{nakajimaLecturesOnHilbertSchemes}.  It is at a
crossroads of several approaches. It is related to the 
framed moduli space of torsion free sheaves on $\PP^2$,
monads and adhm-structures,  quivers of commuting matrices.

Our project was to provide a
 description   in the same vein for 
$\hilbScheme{\PP^n}{p}$, i.e.  
we wanted to replace the constant $p$ by any Hilbert polynomial $p$ 
and the affine plane $\A^2$ by a projective space $\PP^n$ of any
dimension.

An extension of Nakajima's construction 
has been realized by Bartocci, Bruzzo, Lanza and Rava 
\cite{bartocciBruzzoLanzaRava}.
They replace the affine plane $\A^2$ with 
the  total space of $\mathcal O_{\PP^1}(-n)$ and 
   use a description of the 
moduli space parametrising isomorphism classes of framed sheaves on 
the Hirzebruch surface $\Sigma_n$. 
The computations of the paper show that it is not possible to extend the initial
construction directly.  In the sheaf context, 
the  trivialization  at infinity of the sheaf is
responsible for the loss of projectivity. Replacing the surface by a
higher dimensional variety or considering a nonconstant Hilbert
polynomial weakens the link between sheaves and Hilbert schemes.

We may reformulate the above obstructions to extend Nakajima's
construction 
in matrix terms. Recall that 
a zero-dimensional subscheme $Z\subset 
\A^2$ is represented by a pair of commuting matrices $X,Y$ 
corresponding to the multiplication by the variables $x,y$ on 
the vector space $O_Z\simeq \kk^{length(Z)}$, together with a cyclic
vector $v\in \kk^{length(Z)}$ for the pair $(X,Y)$. The matrices 
are determined up to the choice of the base of $O_Z$,
and the cyclic vector is the 
algebraic counterpart of the constant function $1\in O_Z$
generating $O_Z$ as a $k[x,y]$-module. 
Equivalently, the Hilbert scheme is constructed as a quotient of an
open set of a commuting variety, where the commuting variety is 
a moduli space parametrising pairs $(X,Y)$ of commuting matrices. 

When one tries to extend  the construction with commuting matrices from the case of zero-dimensional schemes in $ \mathbb A^2$  to the case of projective schemes   $Z\subset \PP^n$  with any Hilbert polynomial,  the first challenge is that of finding suitable vector spaces of finite dimension, as for instance $H^0(O_Z(t))$ (while the dimension of  $H_*(\mathcal O_Z)$  is  infinite).

The multiplication by  variables $x_i$
yield morphisms
$\mu_i:H^0(O_Z(t))\rightarrow H^0(O_Z(t+1))$ and, if  $t$ is chosen  larger 
than  or equal to  the Castelnuovo-Mumford regularity of $Z$, these maps contain much information on $Z$. 
However, the source space and the target space are different and the
commutativity $\mu_i\mu_j=\mu_j\mu_i$ does not make sense.  Indeed, when  $\p$ is nonconstant, the underlying matrices   $\mathcal M_i$  are not 
square matrices and their sizes are incompatible; when $p$ is
constant,  the matrix sizes are compatible 
but we miss a trivialization at infinity to identify 
$H^0(O_Z(t))$ with $H^0(O_Z(t+1))$. Indeed, in the projective case, there is 
no privileged element in $H^0(O_Z(t))$ and no natural cyclic vector notion.

The above analysis shows that for a construction of $\hilbScheme{\PP^n}{\p}$
based on the multiplicative action of the variables,  we require  a framework
where we can formulate substitutes for the
commutativity and  the cyclic conditions.

In the first part of the paper, 
we introduce a quiver and we formulate these substitutes as
technical conditions on the representations  
of  the quiver that we consider. 
We proceed as follows.

We choose any integer number $R$ larger than or equal to the  Gotzmann  number $\rg$ of  $p$ and we consider the quiver $\Q_\p$
with 4 vertices, $2n+3$ arrows, dimension vector $(\binom{R -1 + n}{ R-1},\binom{R+n}{  R},\p(R),\p(R+1))$ and corresponding vector spaces
$S_{R-1},S_{R},\kk^{\p(R)},\kk^{\p(R+1)}$, where
$S:=\kk[x_0,\dots,x_n]$.

%

\begin{center}
\begin{picture}(0,0)%
\includegraphics{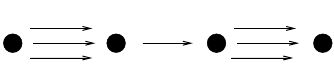}%
\end{picture}%
\setlength{\unitlength}{2072sp}%
\begingroup\makeatletter\ifx\SetFigFont\undefined%
\gdef\SetFigFont#1#2#3#4#5{%
  \reset@font\fontsize{#1}{#2pt}%
  \fontfamily{#3}\fontseries{#4}\fontshape{#5}%
  \selectfont}%
\fi\endgroup%
\begin{picture}(5070,1050)(1156,-2506)
\put(1801,-2446){\makebox(0,0)[lb]{\smash{{\SetFigFont{6}{7.2}{\rmdefault}{\mddefault}{\updefault}{\color[rgb]{0,0,0}$\mu_n$}%
}}}}
\put(1846,-1636){\makebox(0,0)[lb]{\smash{{\SetFigFont{6}{7.2}{\rmdefault}{\mddefault}{\updefault}{\color[rgb]{0,0,0}$\mu_0$}%
}}}}
\put(4861,-1591){\makebox(0,0)[lb]{\smash{{\SetFigFont{6}{7.2}{\rmdefault}{\mddefault}{\updefault}{\color[rgb]{0,0,0}$M_0$}%
}}}}
\put(4861,-2446){\makebox(0,0)[lb]{\smash{{\SetFigFont{6}{7.2}{\rmdefault}{\mddefault}{\updefault}{\color[rgb]{0,0,0}$M_n$}%
}}}}
\put(3421,-1861){\makebox(0,0)[lb]{\smash{{\SetFigFont{6}{7.2}{\rmdefault}{\mddefault}{\updefault}{\color[rgb]{0,0,0}$\rho$}%
}}}}
\put(1171,-1771){\makebox(0,0)[lb]{\smash{{\SetFigFont{6}{7.2}{\rmdefault}{\mddefault}{\updefault}{\color[rgb]{0,0,0}$S_{R-1}$}%
}}}}
\put(2791,-1816){\makebox(0,0)[lb]{\smash{{\SetFigFont{6}{7.2}{\rmdefault}{\mddefault}{\updefault}{\color[rgb]{0,0,0}$S_R$}%
}}}}
\put(4366,-1861){\makebox(0,0)[lb]{\smash{{\SetFigFont{6}{7.2}{\rmdefault}{\mddefault}{\updefault}{\color[rgb]{0,0,0}$k^{p(R)}$}%
}}}}
\put(5941,-1816){\makebox(0,0)[lb]{\smash{{\SetFigFont{6}{7.2}{\rmdefault}{\mddefault}{\updefault}{\color[rgb]{0,0,0}$k^{p(R+1)}$}%
}}}}
\end{picture}%

\end{center}

 Then we consider the representations 
$\mu_0,\dots,\mu_n,\rho,M_0,\dots,M_n$ of the quiver such that:
\begin{itemize}
\item The map $\mu_i$ is the multiplication by the variable $x_i$. 
\item The map $\rho$ is surjective
\item The images of the $M_i$ satisfy the condition $Im(M_0)+\dots
  +Im(M_n)=\kk^{\p(R+1)}$. 
\item $M_i \circ \rho\circ \mu_j = M_j \circ\rho \circ \mu_i$ for every $i,j\in
  \{0,\dots,n\}$. 
\end{itemize}

There is a functor    associated to these representations,  which is the functor of points $\underline{C^p}$  of a scheme  $C^p$.   There is an action of
$\gl{p(R)}\times \gl{p(R+1)}$ on  $C^p$ corresponding to the base
changes on the last two vertices of the quiver. Our description of
the Hilbert scheme is summarized in  
the following theorem.  
  
\begin{theoremA} 
 $C^p$ is a $\gl{p(R)}\times \gl{p(R+1)}$ principal bundle over
  the Hilbert scheme $\hilbScheme{\PP^n}{p}$. 
\end{theoremA}

The theorem provides a new universal property
for the Hilbert scheme: it is possible to describe locally a family
of subschemes of $\PP^n$  
using families of matrices from the quiver description, up to action 
of  a group. 
Describing schemes in terms of linear
algebra up to action may be more convenient
than the usual description in terms of polynomial ideals 
(see \cite[Prop. 3.14]{buloisEvain} for an explicit example).

Recall that 
Grassmannians are quotients of Stiefel varieties, 
and that Pl\"ucker coordinates are computable from Stiefel coordinates 
\cite{gelfand-kapranov-zelevinsky94:multidemensionalDeterminants}.
In our context, the ``Stiefel'' coordinates  on $C^p$  are the entries 
of the matrices $\rho,M_0,\dots,M_n$. In section \ref{sec:Plucker coordinates} we describe 
 the Pl\"ucker coordinates of the Hilbert scheme  in  terms of  these Stiefel
coordinates of $C_p$ (Proposition \ref{Plcoordinates+1}):
\begin{itemize} 
  \item the maximal minors of $\ro$ give  Pl\"ucker coordinates for the embedding  $\hilbScheme{\PP^n}\p \hookrightarrow\GrassScheme{S_{R}}{\p(R)}$;
  
 \item  the  maximal minors of $\sum_{i=0}^n ( M_i \circ\ro )\colon S_R ^{n+1} \rightarrow \kk^{p(R+1)}$  give Pl\"ucker coordinates for the embedding  $\hilbScheme{\PP^n}\p  \hookrightarrow  \GrassScheme{S_{R+1}}{\p(R+1)}$.  
\end{itemize}

The notations to formulate the main results about the equations for the Hilbert scheme are the following.  We consider exterior products of the type 
\begin{equation}\label{for:extprod}\ell z _1\wedge \dots \wedge  \ell z_{b } \wedge v_{b+1}\wedge \dots  \wedge  v_{\p(R+1)}\end{equation}
   where  $z_i\in S_{R},  v_j \in S_{R+1}$ are
monomials and $b\leq \p(R+1)+1$;  note that \eqref{for:extprod}  makes
sense only if  $b\leq \p(R+1)$, but by convention we set that they
are identically  zero for  $b=\p(R+1)+1$, so that the case $b=\p(R)+1$
makes sense also for a constant Hilbert polynomial.

If we chose as  $\ell$  a variable $x_i$, \eqref{for:extprod} corresponds to  a Pl\"ucker
coordinate on the Grassmannian $\GrassScheme{S_{R+1}}{\p(R+1)}$.
If $\ell$ is a linear form in $S_1$,  the multilinear expansion of \eqref{for:extprod} gives  a linear combination of  Pl\"ucker
coordinates.
If   $\ell$ is the  \lq\lq  generic\rq\rq\  linear form 
\noindent $L=y_0x_0+\dots+y_nx_n$  with indeterminate coefficients $y_i$, 
the multilinear expansion of \eqref{for:extprod} gives a  homogeneous polynomial   of degree $b$ in the variables $y_i$  and  linear combinations of  Pl\"ucker coordinates on  $\GrassScheme{S_{R+1}}{\p(R+1)}$ as coefficients. Let  $m=y_{i_1} \cdots y_{i_b} $ be any such monomial, 
 $\underline x$ be  the tuple      $(x_{i_1} , \dots , x_{i_b}) $ 
and    
$ \underline z$, $\underline v$ be  the tuples of monomials $z_i$, $v_j$.
We denote the linear combination of Pl\"ucker coordinates
 which  is  the coefficient of  $m $ by \noindent $F(\underline x,\underline z,\underline v)$   when $b=p(R)$ and by 
 $E(\underline x,\underline z,\underline v)$  when $b=\p(R)+1$.
 
 Since,  up to a sign, permutations on the lists $\underline x,\underline
z,\underline v$ do not  modify $E(\underline x,\underline
z,\underline v)$, we assume each list  ordered in increasing order  (for instance w.r.t. the lexicographic order).

\begin{theoremB} Let    $\p$ be  any  Hilbert polynomial of subschemes
  of $\PP^n$ and let $r$ be the Gotzmann number of $\p$. 
For  any     positive integer $R$, consider   the
  Pl\"ucker 
  embedding  $\GrassScheme{S_{R+1}}{\p(R+1)}\hookrightarrow \PP^{D(R+1)}$ and the following three sets of equations on $\PP^D$:
\begin{enumerate}[1)]
\item the quadratic Pl\"ucker relations of the Grassmannian,
\item the linear equations $E(\underline x,{\underline z},{\underline v}
  )=0$ ( non trivial only for a non-constant $p$ )
\item the quadratic equations \noindent $F(\underline x,\underline {z},\underline {v} )F(\underline x', \underline {z'},\underline {v'})-
 F(\underline x,\underline {z'},\underline {v})F(\underline x',\underline {z},\underline {v'})=0$
\end{enumerate}

  If   $R \geq r $,     then the image of the  embedding  $j_{R+1}\colon \hilbScheme{\PP^n}\p \hookrightarrow  \PP^{D(R+1)} $ is defined by the linear and quadratic equations above.
\end{theoremB}

We underline that our result does not apply to the minimal standard
embedding  $j_{r}$.   In fact, we prove that (except for few very
trivial cases) these equations considered in the case  $R=r-1$ 
define a subscheme of  $\PP^{D(r)}$  that properly contains the Hilbert scheme (Proposition \ref{semifinale}). Furthermore,  we explicitly present a  Hilbert scheme whose image under the  minimal embedding $j_r$ cannot be cut out by any set of equations of degree $\leq 2$  (Example \ref{nonbasta2}). 

 Even though this could appear as a striking reason to motivate the
 non-applicability of Theorem B to the embedding $j_r$, in the final
 section  we present  a different, deeper  motivation for this
 apparent failure,  showing that our equations  have an interesting
 meaning also for every $R\leq r$. 
We show that the problems are concentrated in points with high
regularity and that our equations define properly define the Hilbert scheme 
on the locus of points with adequate regularity.

More precisely, let us denote by $\hilbScheme{\PP^n}{p, [r']}$ the open subscheme of  $\hilbScheme{\PP^n}{p}$
 parametrising subschemes with Hilbert polynomial $p$ and regularity
 at most $r'$. It is proved in \cite{BBR}  that for every   $s\geq r'$
 there is a closed   embedding $j_{r',s}\colon \hilbScheme{\PP^n}{p,
   [r']}\hookrightarrow \PP^{D(s)} \setminus L^{r',s}_{p}$,  where
 $L^{r',s}_{p}$ is a suitable linear subspace of $\PP^{D(s)}$.
   
    In Theorem \ref{reglim} we complete Theorem B proving that: 
    \smallskip

   { \it   If   $r'$ is any positive integer  lower than or equal to $  r$, and   $s\geq r'+1$,  the image of $j_{r',s}\colon \hilbScheme{\PP^n}{p, [r']}\hookrightarrow \PP^{D(s)}\setminus L^{r',s}_{p}$   is given by the linear and quadratic equations 1),2),3)  of Theorem B.}

Note that the embedding $j_{r',s}$ is defined for every $s\geq r'$, but our equations define its image only if $s\geq r'+1$.
\ecr

\subsection*{Overview of the proof of Theorem B}
\label{sec:overview-proof}

The standard way to find  equations for $\hilbScheme{\PP^n}\p$ as a
subscheme of a Grassmannian that we can find in literature is the
following. One chooses a degree $R$ larger or equal to the Gotzmann
number of $\p$,   a subspace  $V\subset S_R$ of codimension $\p(R)$
and    looks at its \lq\lq expansion \rq\rq\ $ S_1V$ in the next  degree $R+1$. By Gotzmann's persistence   (Theorem \ref{macaulayMaximalGrowth}  \eqref{macaulayMaximalGrowth_1}) $V$ corresponds to a point of $\hilbScheme{\PP^n}\p$ if and only if the dimension of $S_1V$ is the minimum allowed by Macaulay's growth    (Theorem \ref{macaulayMaximalGrowth}  \eqref{macaulayMaximalGrowth_2}).  

\medskip

In this paper we follow a different approach, that in some sense goes
in the opposite direction. We consider a subspace  $W$ in $ S_{R+1}$
of codimension $\p(R+1)$  and look at the previous degree $R$.  For
subspaces $W$ corresponding to points of the Hilbert scheme,  $(W\colon S_1)$  has codimension   $\p(R)$ in $S_R$ and,   according to  our construction,  its  Pl\"ucker
coordinates    are maximal minors of a map
   $( M_0 \circ\ro,\dots,
 M_n\circ \rho)$  which is a composition (Proposition \ref{Plcoordinates+1}).

If   the 
dimension  of the space $F$ in the middle of a composition $E\to F \to
G$ is too small,  the maximal minors  vanish. In our context, 
this happens  if  in \eqref{for:extprod} we choose $b= \p(R)+1$. After some algebraic manipulation this leads to 
the linear equations   $E(\underline x,\underline z,\underline v)=0$.
  
   \medskip
   
  These   linear equations define a   subscheme $\mathbf E$ of the Grassmannian that contains the Hilbert scheme, but in general does not coincides with it.
In \S  \ref{sec:6} we  give  an intrinsic description   of $\mathbf E$ as the locus of points $W\in  \GrassScheme{S_{R+1}}{\p(R+1)}$  
 such that for $\ell$ general in $S_1$ the codimension of $(W\colon \ell)$ is $p(R)$,  the maximum  allowed by  Green's hyperplane
restriction  theorem   (Theorem \ref{macaulayMaximalGrowth}  \eqref{macaulayMaximalGrowth_3} and  Theorem \ref{prop:generalform1}).

 \medskip

Finally, in  \S  \ref{sec:equations-degree-twoRED} we prove that    $\hilbScheme{\PP^n}\p$  is the locus of points
$W \in \mathbf  E$  such that  for  $\ell $ general   $(W:\ell )$  does not depend on $\ell$, hence  it coincides with $(W:S_1)$  
(Theorem \ref{prop:generalform}).
We conclude the  proof of Theorem B  showing that the quadratic equations 
\noindent $F(\underline x,\underline {z},\underline {v} )F(\underline x', \underline {z'},\underline {v'})=
 F(\underline x,\underline {z'},\underline {v})F(\underline x',\underline {z},\underline {v'})$
 are fulfilled at a point $W $ in $\mathbf  E $ exactly when  for $\ell$  general   the \Pl\ coordinates   of $(W:\ell )$ in  $ \GrassScheme{S_{R}}{\p(R)}$   do not depend on $\ell$.

  \medskip

The proof we just outlined is developed in the course of the paper in a more general framework, not only for $\kk$-points but  for  families, so that the  equations we obtain define schematically the Hilbert scheme.

\bigskip

\thanks{ A workshop ``Components of Hilbert Schemes'' was organized by the American Institute of Mathematics
  from July 19 to July 23, 2010. This is the place where the authors
  met for the first time. We thank the institute and the organizers.
  
  We thank Steve Kleiman and Michel Brion for their useful comments.}

\section{Generalities  and Embeddings of the Hilbert scheme}
\label{sec:embedd-hilb-scheme}

In this section, after some general notation, we recall some of the classical material used to embed 
Hilbert schemes into Grassmannians. 

\medskip

  We work over a  field $\kk$ 
of any characteristic; in sections from 2 to 7 we  assume that it is algebraically closed;    in the final section we will prove that our equations are valid on any field.

\smallskip

  Let $S=\kk[x_0,\dots,x_n]$ and
  $S_A=A[x_0,\dots,x_n]$ for any $\kk$-algebra $A$.  We denote by
  $S_d$ and $S_{A,d}$ the free submodules of homogeneous polynomials  of degree 
  $d$ and by  $N(d)$ their dimension.  We denote by the same letter $\mhu{i}:S\fd S$ and $\mhu{i}:S_A\fd
  S_A$ the multiplication by the variable
  $x_i$. 

 When $A$ is a field, we  often consider the $\kk$-vector spaces $S_1\simeq \kk^{n+1}$
and    $S_{A,1}\simeq A^{n+1}$
as   topological spaces endowed with the Zariski topology. 
We  say that a property holds for a general linear form $\ell$ in $S_1$ (resp. $S_{A,1}$)  if it holds for every $\ell$ in a  non-empty Zariski open subset of  $S_1$ (resp. $S_{A,1}$).
We will use the following well known facts.

\begin{lemma}   \label{ZariskiDense}  Let $A$ be any  $\kk$-algebra
  ($k$ algebraically closed, but infinite is sufficient ).  
\begin{enumerate}
\item \label{ZariskiDense1}  If  $F\in A[y_1, \dots,  y_m ]$   vanishes on  a dense subset of  $\kk^m$, then it is  the null polynomial. 

\item \label{ZariskiDense2}   
 If $A$ is a field, the Zariski topology of $\kk^m$ is the subspace  topology induced by the Zariski topology of $A^m$.
\end{enumerate}
\end{lemma}
\begin{proof}
We prove that 
 the set of zeros $Z$ in $\kk^m$ of  a polynomial $F\in A[y_1, \dots, y_m]$  is also the set of common zeros of  suitable polynomials $G_i \in \kk[y_1, \dots, y_m]$.
If $a_1, \dots, a_v\in A$ are  the coefficients of $F$, it is sufficient to prove the result  assuming that $A$ is   the finitely generated $\kk$-algebra $\kk[a_1, \dots, a_v]$.  Then, by Noether's Normalization Lemma, $A=C[b_1, \dots, b_s]$ with $C:=\kk[T_1, \dots, T_r]$ polynomial ring in the indeterminates $T_i$ and $b_i$ integral over $C$.

Then, it is sufficient to prove the result for polynomials with coefficients in  $B[t]$ assuming that  the result holds for polynomials with coefficients in a $\kk$-algebra $B$ and that $t$ is either integral or transcendent over $B$. 
In both cases, the coefficients in $B[t]$ of a polynomial $F\in
B[t][y_1, \dots, y_m]$ are contained in some  free $B$-module of
finite rank $d$  with basis given by  powers of $t$: $B[t]$ itself, if
$t$ is integral,
the $B$ module generated by the powers of $t$ up to the maximum appearing in $F$, if   $t$ is transcendent.       Then $F=\sum_{i=0}^{d-1} G_i t^i$ with $G_i\in B[y_1, \dots, y_n] $ and  $Z$ is 
the set of  the common zeros of   the polynomials  $G_i$.

We have proved the second point. The first one is now an easy
induction on the number $m$ of variables. 
\end{proof}

For any $\kk$-algebra $A$ and $A$-module $W$ in $S_{A,d}$, we will denote by $S_{1}W  $ the $A$-submodule of $S_{A,d+1}$ 
generated by the images of the multiplication maps. Moreover, for every linear form  $\ell \in S_{1}$ we will denote by  $ (W:\ell)$ 
  the $A$-module   $\{f\in S_{A,d-1} \ | \ \ell f\in W\}$.

\smallskip

By simplicity for every tuple  $\underline w=(w_1,\dots, w_b)$  of  elements of an $A$-module $M$ we shortly write 
  $\wedge  \underline w$ to denote the element  $w_1\wedge \dots \wedge w_b$ in $\wedge^b M$. 

\smallskip
 The vector spaces $S_{R}$ and $S_{R+1}$ are considered
with their natural bases of monomials ordered in some way (for instance  lexicographically); we say that   tuple of monomials $(w_1,\dots, w_b)$ 
is ordered increasingly if $w_1\leq w_2 \leq \dots \leq w_b$.

\smallskip
 If  $\p$ is the Hilbert polynomial of a subscheme $Z\subset \PP^\n$,
the Gotzmann number of $\p$ is  the Castelnuovo-Mumford regularity of $\p$, i.e. the smallest integer $m$ such that 
every $Z \subset \PP^{\n}$ with Hilbert polynomial $\p$ is
$m$-regular \cite[Proposition  C.24]{iarrobino-kleiman:around_gotzmann_number}. In particular the Hilbert function $H_Z$ of $Z$ satisfies
$H_Z(d)=p(d)$   for  every $d\geq \rg$.
  Note that $\rg$ depends on  $\p$, but neither  on the ground field nor on $n$.

From now on, $\p$ will denote a Hilbert polynomial for subschemes of $\PP^\n$,   $r$  its Gotzmann number 
and $R$ any number $ \geq \rg$. Moreover, for every integer $t$ we will denote by $q(t)$ the number $N(t)-\p(t)$.

\begin{theorem}
\label{macaulayMaximalGrowth} 
  Assume that the  $\kk$-algebra $A$ is a  field.  Let  $W $ be a vector space in $S_{A,d}$ with  $\Codim_A( W,S_{A,d})=\p(d)$ and $d\geq r$.   Then, 
\begin{enumerate}
\item (Macaulay) \label{macaulayMaximalGrowth_1}  $\Codim_A( W  S_{A,1} ,S_{A,d+1})\leq \p(d+1)$.
\item (Gotzmann) \label{macaulayMaximalGrowth_2}  The equality $\Codim_A( W  S_{A,1} ,S_{A,d+1})= \p(d+1)$ holds  if and only if the Hilbert polynomial 
of the ideal generated by $W$ is $\p$. 
\item (Green)\label{macaulayMaximalGrowth_3}  If   $d\geq r+1$ and  $\ell $ is  general  in $ S_{1}$, then  $\Codim_A( (W:\ell),S_{A,d-1})\geq \p(d-1)$. 
\end{enumerate}
\end{theorem}
\begin{proof}
 \eqref{macaulayMaximalGrowth_1}  is  a consequence of  Macaulay's  theorem on the growth of the Hilbert functions and \eqref{macaulayMaximalGrowth_2}  is  
Gotzmann's persistence theorem \cite{gotzmann}. These results can be found in several research papers and books;  a version  very close to ours for  
notation and intent is that of \cite[Proposition 4.2]{haiman_sturmfels02:multigradedHilbertSchemes}. 

\medskip
  
      \eqref{macaulayMaximalGrowth_3}   follows by  Green's hyperplane
restriction theorem  proved in \cite[Theorem 1]{green:macaulayGotzmann};  in fact,  if $c=\p(d)$ and $d$ is larger than the Gotzmann number of $\p$, then the number  that in \cite{green:macaulayGotzmann} is  denoted as $c_{\langle d\rangle}$ coincides with $\p(d)-\p(d-1)$.  We also refer to   \cite[Theorem 4.2.12]{BrunsHerzog}) where
 it is clearly stated that this result only needs that the ground field is infinite.    Note  that   in the quoted paper the result is proved  for a general $\ell$ in $S_{A,1}$, 
hence  it holds  for a general $\ell$ in   in $S_1$ (Lemma
\ref{ZariskiDense} \eqref{ZariskiDense2}). 
\end{proof}

Exploiting Theorem \ref{macaulayMaximalGrowth}    the following result
 realizes $\hilbScheme{\PP^\n}{p}$ as a closed subscheme of the
   product of Grassmannians  $\GrassScheme{S_{R}}{p(R)}\times \GrassScheme
   {S_{R+1}}{p(R+1)}$  \cite[Bemerkung 3.2]{gotzmann},\cite[Proposition C.28, Theorem
   C.29]{iarrobino-kleiman:around_gotzmann_number}, \cite{haiman_sturmfels02:multigradedHilbertSchemes}, \cite[Exercise VI-3]{EisenbudHarris}.

Following \cite{haiman_sturmfels02:multigradedHilbertSchemes}, 
We will denote by $\underline Y$ the functor of points of a scheme $Y$
from $\kk$-algebras to sets.

\begin{theorem} \label{constructionHilbertDegreeRRPlusOne}
 The Hilbert scheme  $\hilbScheme{\PP^\n}{p}$ is the subscheme of   $\GrassScheme{S_{R}}{p(R)}\times \GrassScheme
   {S_{R+1}}{p(R+1)}$ whose functor of points from $k$-algebras to sets  is  given by 
   
  $\hilbSchemeFunctor{\PP^n}{p}(A)=\{(I_{A,R},I_{A,R+1})\subset  (
  S_{A,R},  S_{A,R+1})$ that satisfy the  following conditions $\}$
  \begin{itemize}
   \item $ S_{A,R}/I_{A,R}$ is locally free of rank $\p(R)$
   \item    $ S_{A,R+1}/I_{A,R+1}$ is   locally free of rank  $\p(R+1)$
   \item  $x_i I_{A,R}\subset I_{A,R+1}$ for each variable $x_i $. 
   \end{itemize}

Moreover, the first  and  second projections give the  embeddings   $j_R \colon  \hilbScheme{\PP^\n}{p} \hookrightarrow \GrassScheme{S_{R}}{p(R)}   $  and  $j_{R+1}\colon  \hilbScheme{\PP^\n}{p}\hookrightarrow  \GrassScheme{S_{R+1}}{p(R+1)}   $. 
\end{theorem}

\section{A new  description    of the Hilbert scheme}
\label{sec:constr-scheme-c}

\begin{notation}\label{not2}
  If $\phi_j:E\fd F$, for  $j=0, \dots, n$,    are  morphisms of
  $A$-modules and $ B$ is an $A$-algebra,
  we will use the following notations  
\begin{itemize} 
\item $\phi_j \otimes_A B:E\otimes_A B \fd F\otimes_A B$ is  the
  morphism of modules with $(\phi_j\otimes_AB)(e\otimes
  b)=\phi_j(e)\otimes b$,
\item $\phi $ is  the list    $(\phi_0, \dots, \phi_n)$ ,
\item $\oplus \phi $ is the morphism $ E\oplus \dots \oplus E \fd F \oplus
  \dots \oplus F $   given by    $\oplus  \phi(e_0, \dots, e_n)=(\phi_0(e_0), 
\dots, \phi_i(e_n))$,
\item   $  \Sigma \phi$ is the morphism $ E\oplus \dots \oplus E \fd F  $   given by   $\Sigma  \phi(e_0, \dots, e_n)=\Sigma_{j=0}^n\phi_j(e_j) $.
\end{itemize}
\end{notation}

\begin{remark}\label{condiagrammaH}
By the functorial description of the Hilbert scheme given in Theorem \ref{constructionHilbertDegreeRRPlusOne},  every map $\spec{A}\rightarrow 
\hilbScheme{\PP^\n}{p}$ corresponds to a  commutative  diagram with exact rows.


  \begin{equation}\label{diagrammaH}
    \begin{array}{cccccccccc}
     &&& & (S_{A,R-1})^{\n+1} & & & & \\
                     & &&& \big\downarrow {(  \oplus{\mu} )_{R-1}} & & & & & \\
0&\rightarrow&  (I_{A,R})^{\n+1} &\hookrightarrow  & (S_{A,R})^{\n+1} &
& {\stackrel{ \pi_R \oplus \dots \oplus \pi_R  }\longrightarrow} &
(S_{A,R}/I_{A,R})^{\n+1}  &\fd & 0 \\
    &&  \big\downarrow (\Sigma\mu)_{R,I} & & \big\downarrow (\Sigma\mu)_R && &
      \big\downarrow (\Sigma  \overline{\mu})_R && \\
 0&\rightarrow &I_{A,R+1} & \hookrightarrow & S_{A,R+1} & &   {\stackrel{ \pi_{R+1}}\longrightarrow}  &
S_{A,R+1}/I_{A,R+1} &\fd & 0   \\ 
    \end{array}
    \end{equation}
  where  $ \mu$ is the list $(\mhu{0}, \dots, \mhu{n})$ with
  $\mhu{i}:S_A\rightarrow S_A$  the multiplication by the variable
  $x_i$, $\pi_R$, $\pi_{R+1}$ are the projections on the quotients,
  $\overline \mu=(\overline \mu_{0}, \dots, \overline \mu_{n})$ is the
  list of quotient maps, 
  $\Sigma \mu$ and $\Sigma \overline \mu$ are defined by notation
  \ref{not2},  and  $(\Sigma \mu)_{R}$,$(\Sigma \mu)_{R,I}$,  $(\Sigma
  \overline \mu)_{R}$ are defined in the natural way by restrictions of
  $\Sigma \mu$ and $\Sigma \overline \mu$, the indices keeping track
  of the domain and codomain. 
\end{remark}

To build the variety $C^p$ above the Hilbert scheme we elaborate on the above diagram and  construct a  functor of representations of the   quiver $Q_p$ of the introduction

\begin{center}
\begin{picture}(0,0)%
\includegraphics{quiver.pdf}%
\end{picture}%
\setlength{\unitlength}{2072sp}%
\begingroup\makeatletter\ifx\SetFigFont\undefined%
\gdef\SetFigFont#1#2#3#4#5{%
  \reset@font\fontsize{#1}{#2pt}%
  \fontfamily{#3}\fontseries{#4}\fontshape{#5}%
  \selectfont}%
\fi\endgroup%
\begin{picture}(5070,1050)(1156,-2506)
\put(1801,-2446){\makebox(0,0)[lb]{\smash{{\SetFigFont{6}{7.2}{\rmdefault}{\mddefault}{\updefault}{\color[rgb]{0,0,0}$\mu_n$}%
}}}}
\put(1846,-1636){\makebox(0,0)[lb]{\smash{{\SetFigFont{6}{7.2}{\rmdefault}{\mddefault}{\updefault}{\color[rgb]{0,0,0}$\mu_0$}%
}}}}
\put(4861,-1591){\makebox(0,0)[lb]{\smash{{\SetFigFont{6}{7.2}{\rmdefault}{\mddefault}{\updefault}{\color[rgb]{0,0,0}$M_0$}%
}}}}
\put(4861,-2446){\makebox(0,0)[lb]{\smash{{\SetFigFont{6}{7.2}{\rmdefault}{\mddefault}{\updefault}{\color[rgb]{0,0,0}$M_n$}%
}}}}
\put(3421,-1861){\makebox(0,0)[lb]{\smash{{\SetFigFont{6}{7.2}{\rmdefault}{\mddefault}{\updefault}{\color[rgb]{0,0,0}$\rho$}%
}}}}
\put(1171,-1771){\makebox(0,0)[lb]{\smash{{\SetFigFont{6}{7.2}{\rmdefault}{\mddefault}{\updefault}{\color[rgb]{0,0,0}$S_{R-1}$}%
}}}}
\put(2791,-1816){\makebox(0,0)[lb]{\smash{{\SetFigFont{6}{7.2}{\rmdefault}{\mddefault}{\updefault}{\color[rgb]{0,0,0}$S_R$}%
}}}}
\put(4366,-1861){\makebox(0,0)[lb]{\smash{{\SetFigFont{6}{7.2}{\rmdefault}{\mddefault}{\updefault}{\color[rgb]{0,0,0}$k^{p(R)}$}%
}}}}
\put(5941,-1816){\makebox(0,0)[lb]{\smash{{\SetFigFont{6}{7.2}{\rmdefault}{\mddefault}{\updefault}{\color[rgb]{0,0,0}$k^{p(R+1)}$}%
}}}}
\end{picture}%
\end{center}

\begin{definition} \label{def:cp-functorial}
  Let  $A$ be a  $\kk$-algebra.  Let $\mathfrak C^p(A)=\{(\mu ,\rho,\M{})\}$ 
  where:
 \begin{itemize}
 \item $\mhu{}=(\mhu{0}, \dots, \mhu{n})$ and $\mhu{i}:S_{A,R-1}\fd
  S_{A,{R}}$ is the multiplication by the variable
  $x_i$,
 \item $\M{}=(\M{0}, \dots, \M{n})$ and  $\M{i}:A^{\p(R)}\fd
  A^{\p(R+1)}$ is a morphism of $A$-modules,
  \item  $\Sigma\M{}:(A^{\p(R)})^{\n+1}\fd
  A^{\p(R+1)}$   is surjective
\item $\ro:  S_{A,R} \fd A^{\p(R)}$ is a surjective
morphism of $A$-modules,
\item for every pair $(i,j)\in \{0,\dots,n\}$,  $\M{i}\circ \ro \circ
  \mhu{j}=\M{j}\circ \ro \circ
   \mhu{i}$
  \end{itemize}
\begin{displaymath}
 S_{A,R-1}\stackrel{\mhu{i}}{\fd} S_{A,R}  \stackrel{\rho}{\fd}
  A^{\p(R)} \stackrel{\M{j}} {\fd} A^{\p(R+1)}.
\end{displaymath}

Since the tensorisation preserves the surjectivity, for any map of
$\kk$-algebras $A\fd B$, we have a morphism $\mathfrak C^p(A)\fd
\mathfrak C^p(B)$ which sends $(\mu, \rho,  M)$ to $(\mu\otimes_A
B,\rho \otimes_A
B, M\otimes_A B)$. This makes $\mathfrak C^p$ a functor from the category of
$\kk$-algebras to the category of sets. 
\end{definition}

\begin{remark}
  The set $\mathfrak C^p (A)$ and the map $\mhu{}$ depend on $R$, but for brevity, $R$
  is not included in our notation. Similarly, we will use the notation 
  $\mathfrak C^p(A)=\{(\rho, M)\}$ as a shortcut for
  $\mathfrak C^p(A)=\{(\mhu,\rho, M)\}$ 
since there is only one possible choice for
  $\mhu{}$. 
\end{remark}

\begin{proposition} 
  There exists a scheme $C^p$ such that:
  \begin{listecompacte}
  \item $\mathfrak C^p(A)=\underline C^p$, namely $\mathfrak C^p(A)=Hom(\spec{A},C^p)$ for every $k$-algebra $A$;
  \item the $\kk$-points of $C^p$ are representations of the quiver $Q_p$. 
  \end{listecompacte}
\end{proposition}
\begin{proof}
The 
non-trivial fact is the first item. It follows immediately that 
the $k$-points are representations of $Q_p$.

Let $\widetilde {\mathfrak C^p}$ be the extension of  $\mathfrak C^p$ to
the category of $k$-schemes, i.e. $\widetilde {\mathfrak C^p}(Z)=\{(\mu ,\rho,\M{})\}$ 
  where:
 \begin{listecompacte}
 \item $\mhu{}=(\mhu{0}, \dots, \mhu{n})$ and
   $\mhu{i}:S_{R-1}\otimes \mathcal O_Z \fd
  S_{R}\otimes \mathcal O_Z$ is the multiplication by the variable
  $x_i$,
 \item $\M{}=(\M{0}, \dots, \M{n})$ and  $\M{i}:\mathcal O_Z^{p(R)}\fd
  \mathcal O_Z^{p(R+1)}$ is a morphism of $\mathcal O_Z$-modules,
  \item  $\Sigma\M{}:(\mathcal O_Z^{p(R)})^{\n+1}\fd
  \mathcal O_Z^{p(R+1)}$   is surjective
\item $\ro:  S_{R}\otimes \mathcal O_Z \fd \mathcal O_Z^{p(R)}$ is a surjective
morphism of $\mathcal O_Z$-modules,
\item for every pair $(i,j)\in \{0,\dots,n\}$,  $\M{i}\circ \ro \circ
  \mhu{j}=\M{j}\circ \ro \circ
   \mhu{i}$. 
  \end{listecompacte}
It suffices to prove that $\widetilde {\mathfrak C^p}$ is representable to 
obtain the first item of the proposition. 

  Consider the functor $\mathcal V$ from schemes to sets defined as follows.  
If $Z$ is a  $\kk$-scheme, an
  element of $\mathcal V(Z)$ is a couple $(\rho , M)$ where:
  \begin{listecompacte}
  \item $\Sigma \M{} \colon (\mathcal O_Z^{\p(R)})^{\n+1}\fd
  \mathcal O_Z^{p(R+1)}$ is a  morphism of $\mathcal O_Z$-modules,
\item $\ro:   S_{R} \otimes \mathcal O_Z \fd \mathcal O_Z^{p(R)}$ is a 
(possibly not surjective) morphism of $\mathcal O_Z$-modules.
  \end{listecompacte}
For any map of
$\kk$-schemes $\phi:Z_2\fd Z_1$, we have a morphism $\mathcal V({Z_1})\fd
\mathcal V({Z_2})$ which sends $(\rho , M)$ to $(\phi^*\rho ,\phi^*M)$.

For any  finite dimensional 
$k$-vector space $V$, let us denote by $t(V)$ the scheme $\spec{Sym(V^*)}$ and consider the  functor $\underline{t(V)}$ given by 
 $$\underline{t(V)}(Z)=\Hom(Z,\spec{Sym(V^*)})\simeq  \Hom(Sym(V^*), H^0(\mathcal O_Z)) \simeq  H^0(\mathcal O_Z) \otimes_k V$$
  and, for any map of
$\kk$-schemes $Z_2 \fd Z_1$, the map $\underline{ t(V)
}(Z_1)\rightarrow \underline{ t(V) }(Z_2)$
sends $H^0(\mathcal O_{Z_1}) \otimes_\kk V  $ to $H^0(\mathcal O_{Z_2})\otimes_\kk
 V  $  by pullback. 
The above functor $\mathcal V$ is a special case of the functor $\underline{t(V)}$ with 
$V=\Hom((\kk^{\p(R)})^{\n+1},
  \kk^{\p(R+1)}) \oplus \Hom(S_R, \kk^{\p(R)})$.

 We recall  the  notion of relative
 representability from \cite{grothendieck:formalismeFoncteursRepresentables}.  
Let \noindent $F,G$ be functors from the category of $\kk$-schemes 
to sets. Suppose that \noindent $F$ is a subfunctor of \noindent $G$. The inclusion
\noindent $F\subset G$ 
is relatively representable if, for every  $k$-scheme $Z$  and every  morphism of functors $\underline Z\fd G$, the fiber product $\underline Z \times_G
F$ is representable.  Grothendieck,  \cite[Lemme 3.6]{grothendieck:formalismeFoncteursRepresentables} proves that if \noindent $G$ is
representable and if \noindent $F\subset G$ is relatively representable, then 
\noindent $F$ is representable.

In our case,  
 $\widetilde {\mathfrak C^p}$ is a  subfunctor of the representable functor $\mathcal V$  and it is defined by the surjectivity of $\Sigma M$
 and $\rho$, and  by the equality $\M{i}\circ \ro \circ
 \mhu{j}=\M{j}\circ \ro \circ
 \mhu{i}$. Thus it suffices to 
prove  that a subfunctor defined by 
the surjectivity of a morphism of locally free sheaves is relatively representable, 
and that a subfunctor defined by the equality of morphisms of
locally free sheaves is relatively representable. 

The locus in $\spec{A}$ where two matrices $M,N \in \Hom(\spec{A},\kk^{pq})$ of size $p\times q$
with coefficients $m_{ij},n_{ij}$ in $A$ coincide is closed.
More precisely, if $\spec B \fd
  \spec A$ is a morphism, then the pullback matrices $M_B,N_B \in
  \Hom(\spec B,\kk^{pq})$ satisfy $M_B=N_B$ if and only if the morphism  $\spec B \fd
  \spec A$ factorizes through the closed subscheme $Z=\spec {A/J}$ where
  the ideal $J$ is generated by the elements $(m_{ij}-n_{ij})$. 
  It follows that if  $\mathcal F, \mathcal G$ are locally free sheaves on $Z$, and
  if \noindent $F,g\in Hom_{\mathcal O_Z}(\mathcal F, \mathcal G)$ are two morphisms of sheaves, 
there exists a closed subscheme $i_W:W\fd Z$,  such that for all 
$\phi:Y\fd Z$, $\phi^*f=\phi^*g$ iff $\phi$ factorizes through $W$.

Let \noindent $G$ be a functor such that \noindent $G(Z)=\{(f,g,...)\}$, i.e. \noindent $G(Z)$ is a
tuple, and two components \noindent $ f ,g$ of this tuple
correspond to a functorial morphism of
locally free sheaves $\mathcal F_Z \fd \mathcal G_Z$ above $Z$.  Let \noindent $F$ be
the subfunctor of \noindent $G$ defined by the condition \noindent
$ f =g$.  By Yoneda, 
a morphism $\underline Z\fd G$ is defined by an element in \noindent $G(Z)$.  By the above, 
$\underline Z(Y) \times_{G(Y)} F(Y) $ can be identified with
$\underline W(Y)$, where $W$ is the closed subscheme of $Z$ defined by
the condition $f=g$.  Thus $\underline Z \times _G F\simeq \underline W$ and \noindent $F\subset G$ is a relatively representable functor. 
It follows that 
  the condition $\M{i}\circ \ro \circ
  \mhu{j}=\M{j}\circ \ro \circ
  \mhu{i}$ defines a relatively representable (closed) subfunctor of $\mathcal V$.

The fact that the surjectivity condition on a morphism of sheaves  
defines an open subfunctor is a classical argument used in the
construction of the Grassmannians  \cite[Lemme 9.7.4.6]
{ega1Springer}.
Thus $\widetilde{\mathfrak C^p}$ is representable as it is a locally closed
  subfunctor of the representable functor $\mathcal V$. 
\end{proof}

Our next goal is to prove that the Hilbert scheme     is a  quotient
of $C^p$. We first explain how  $C^p$ is related to the description  of $\hilbScheme{\PP^\n}{p}$ given in Theorem \ref{constructionHilbertDegreeRRPlusOne}. We always refer to  Notation \ref{not2}.

\begin{proposition}\label{diagramma} For every $k$-algebra $A$ and $(\ro, M) \in \underline{C^p}(A)$, let  $I_{A,R}:=\Ker(\ro)$,   
$ I_{A,R+1}:=\Sigma \mu(  I_{A,R}^{n+1} )$ and $\Sigma
\mu_{ R,I  }$
 the restriction of $\Sigma \mu$ to  $(I_{A,R})^{n+1} \rightarrow I_{A,R+1}$. 
 
 Then, there is  a morphism  $\beta\colon   S_{A,R+1}  \rightarrow  A^{\p(R+1)} $  such that $I_{A,R+1}=\Ker (\beta)$ and the following  
diagram is  commutative with exact rows:
  \begin{equation}\label{eq:3.2}
    \begin{array}{cccccccccc}
                     &&& & (S_{A,R-1})^{\n+1} & & & & \\
                     & &&& \big\downarrow {\oplus{\mu}} & & & & & \\
0&\rightarrow&  (I_{A,R})^{\n+1} &\hookrightarrow  & (S_{A,R})^{\n+1} &
&\stackrel{\rho\oplus  \dots \oplus \rho  }{\longrightarrow} &
(A^{\p(R)})^{\n+1} &\fd & 0 \\
    &&  \big\downarrow \Sigma\mu_{R,I} & & \big\downarrow \Sigma\mu_R && &
      \big\downarrow \Sigma  M && \\
 0&\rightarrow &I_{A,R+1} & \hookrightarrow & S_{A,R+1} & & \stackrel{\beta}{\rightarrow} &
 A^{\p(R+1)} &\fd & 0   \\ 
    \end{array}.
 \end{equation}
 Moreover,    $(I_{A,R}, I_{A,R+1}) \in \hilbSchemeFunctor{\PP^n}{p}(A)$.
  \end{proposition}
\begin{proof}
We observe that 
\begin{itemize}
\item  $ \rho\oplus\dots \oplus \rho $,  $\Sigma\mhu{R,I}$, $\Sigma \mu_R$ and $\Sigma M$   are  surjective by hypotheses and/or by construction; 
\item by construction the first row   is exact and the square on the left commutes.
\end{itemize}

 We use all these properties in order to define $\beta$  so that also
 the last line is exact and the right square commutes.

We define $\beta$  by diagram chasing in the following way:  by the surjectivity of  $ \Sigma \mu $ every element of  $S_{A,R+1} $ can be 
written (not uniquely) as $\Sigma x_i f_i$ where   \noindent $F:=(f_0,
\dots ,f_n) \in (S_{A,R} )^{n+1}$;  then  we set $ \beta(\Sigma x_i
f_i )= \Sigma M (\rho\oplus\dots \oplus \rho  (f))$.

To verify  that $\beta$  is well defined  it is sufficient to prove
that  when  $\sum x_i f_i=0$  we have $ \Sigma M ( \rho\oplus\dots \oplus \rho (f))=0 $.

 This is obvious if  \noindent $F=( 0,\dots, ,0,f_n)$ , since  
 $\sum x_i f_i=0$ implies  \noindent $ f_n =0$.   Then, we prove the assertion
 for   \noindent $F=( 0 \dots, 0,f_{j-1}, \dots, f_n)$ assuming it holds  for elements of
 the form $( 0 \dots, 0,f_{j}, \dots, f_n)$ 

For  every $ i=j, \dots, n$ we  set $f_i = x_{j-1} f_i'+
f_i''$  with \noindent $ f_i'  \in S_{A,R-1}$ and  $x_{j-1 }$
not appearing in  \noindent $ f_i'' $. The equality 
 $\sum_{i=0}^n x_i f_i=0$ implies \noindent $f_{j-1} +\sum_{i=j}^n  x_i f_i'=0 $ and  $\sum_{i=j}^n x_i f_i''=0 $.  
Then 
$$ \Sigma M (  \rho\oplus\dots \oplus \rho  (f))=\sum_{i=j-1}^{n} M_i ( \rho(f_i))= M_{j-1} ( \rho( f_{j-1})) +\sum_{i=j}^{n} M_i ( \rho(\mu_{j-1}( f_i')))+\sum_{i=j}^{n} M_i ( \rho(f_i'')). $$
The last summand is  equal to $ \Sigma M( \rho\oplus\dots \oplus \rho ((0, \dots, 0 ,f_j'',
\dots, f_n'')))$, hence it vanishes by the inductive assumption.  
Moreover, by the compatibility conditions,  we have 

$ M_i ( \rho(\mhu{j-1}( f_i')))= M_{j-1} ( \rho(\mhu{i}( f_i')))=M_{j-1} ( \rho(x_i f_i'))$.
Therefore $ \Sigma M ( \rho\oplus\dots \oplus \rho  (f))= M_{j-1}(\rho( f_{j-1}+
\sum_{i=j}^{n} x_i f_i') )= M_{j-1}( \rho( 0 ))=0.$ 

The commutativity of the right   square  holds by construction
of $\beta$ and the surjectivity of $\beta$ follows from that of 
 $\rho\oplus\dots \oplus \rho  $, and $ \Sigma  M $.

To complete the construction of our diagram, we now prove
that $ \Ker(\beta)$ is equal to $I_{A,R+1}$. By the commutativity
of the two squares and  the surjectivity of 
$\Sigma\mhu{R,I}$ it follows that $ I_{A,R+1}$ is contained in
$\Ker(\beta)$. 
To prove  the reverse inclusion we observe   that  $ I_{A,R}, I_{A,R+1} ,\Ker(\beta) $
depend functorially on $A$ in the sense that if $A \fd B$ is a
morphism of $\kk$-algebras, if \noindent $L_A\in
\{I_{A,R},\Ker(\beta)\}$ is one of these  two  $A$-modules,  then
\noindent $L_B=L_A\otimes_AB$,  and $I_{B,R+1}$ is the image of
$I_{A,R+1}$ in
$S_{A,R+1}\otimes B$. 
Then, we may check that for each maximal ideal
$\mathfrak m$, $(\Ker(\beta)/I_{A,R+1}))\otimes_A A_{\mathfrak m}=0$.
In other words, we may replace $A$ with $A_\mathfrak m$ and 
suppose that $A$ is local  with maximal ideal
$\mathfrak m$. 

The $A$-module $\Ker(\beta)$ is finitely
generated as a kernel of  a map between finitely generated  free
modules (\cite[Exercise 12, p.32]{atiyahMacdonal:introductionToCommutativeAlgebra}).
Thus $\Ker(\beta)/ I_{A,R+1}$ is finitely generated and, by Nakayama,  
we may even suppose that $A$ is a field. 
When $A$ is a field, the inclusion $  I_{A,R+1} \subset \Ker(\beta)$ is an
equality if $  \Dim\ I_{A,R+1} \geq \Dim\ \Ker(\beta) $ as vector spaces.  Since $\Codim(I_{A,R},S_R)  =\p(R)$,
Macaulay's maximal growth theorem
(Theorem~\ref{macaulayMaximalGrowth} \eqref{macaulayMaximalGrowth_3})  gives the inequality  
$\Codim(I_{A,R+1} ,S_{A,R+1})  
{\leq }\p(R+1)=\Codim(\Ker(\beta),S_{A,R+1})$. 

The final assertion directly follows from   Theorem \ref{constructionHilbertDegreeRRPlusOne}.
\end{proof}

\medskip

Now we are ready to conclude the proof of  the first  main result of the paper.

\begin{proof}  of {\bf Theorem A.}
Our goal is to prove that the Hilbert scheme  $\hilbScheme{\PP^\n}{p}$   is a  quotient
of $C^p$ by  a  natural action of $\gl{\p(R)}\times\gl{\p(R+1)}$. We start with the construction of a morphism $\chi :C^p\fd \hilbScheme{\PP^\n}{p}$. 
  We  consider the description of  $\hilbScheme{\PP^\n}{p}$  given  in Theorem \ref{constructionHilbertDegreeRRPlusOne}.

\medskip

{\bf Claim 1} There exists a surjective morphism $\chi:C^p\fd \hilbScheme{\PP^\n}{p}$. 

\medskip

Making reference to Proposition \ref{diagramma}, we define $\chi$ at the functorial level  by setting  $\chi((\ro, M))=( I_{A,R}=\Ker(\rho) ,  I_{A,R+1}=\Ker(\beta)) \in \hilbSchemeFunctor{\PP^\n}{p}(A)$, 
for every $\kk$-algebra $A$ and  $(\ro,M) \in \underline{C^p}(A)$. By  construction this  association depends functorially on $A$: 
indeed the exactness and commutativity of \eqref{eq:3.2} is preserved by tensorisation since the modules on the right are free.

   \medskip   
   We observe that the difference between the diagram
   \eqref{diagrammaH} associated with the Hilbert scheme and the
   diagram \eqref{eq:3.2} associated with $C^p$ comes from 
   identifications $\tilde\rho:S_{A,R}/I_{A,R}\to A^{\p(R)}$ and
   $\tilde\beta:S_{A,R+1}/I_{A,R+1}\to A^{\p(R+1)}$ such that   $\rho=\tilde\rho\circ \pi_R$
   and $\beta=\tilde\beta \circ \pi_{R+1}$. Starting from
   $\eqref{eq:3.2}$, the isomorphisms $\tilde\rho,\tilde\beta $ are
   obtained by factorization. Starting from \eqref{diagrammaH} and the
   isomorphisms  $\tilde\rho,\tilde\beta $, we will see that it is
   possible to form
 the following diagram  that includes
both \eqref{diagrammaH} and \eqref{eq:3.2}.
 \begin{equation}\label{diagram4}
    \begin{array}{cccccccccccc}
                     &&& & (S_{A,R-1})^{\n+1} & & & & & \\
                     & &&& \big\downarrow {\oplus{\mu}} & & & & & \\
0&\rightarrow&  (I_{A,R})^{\n+1} &\hookrightarrow  & (S_{A,R})^{\n+1} 
&\stackrel{ \pi_R \oplus\dots \oplus \pi_R}{\rightarrow}
                                                  &(S_{A,R}/I_{A,R})^{\n+1}&\stackrel{  \tilde\rho\oplus \dots \oplus\tilde\rho}{\rightarrow}&
(A^{\p(R)})^{\n+1} &\fd & 0 \\
    &&  \big\downarrow \Sigma\mu_{R,I} & & \big\downarrow \Sigma\mu_R && 
      \big\downarrow \Sigma    \overline{\mu}_R  &&  \big\downarrow \Sigma M &\\
 0&\rightarrow &I_{A,R+1} & \hookrightarrow & S_{A,R+1} &  \stackrel{\pi_{R+1}}{\rightarrow} &S_{A,R+1}/I_{A,R+1}&\stackrel{\tilde\beta}{\rightarrow} &
 A^{\p(R+1)} &\fd & 0   \\ 
    \end{array}
 \end{equation}

  It remains to prove that $\chi$ is surjective.  Let   $ \phi:
  U=\spec{A} \hookrightarrow  \hilbScheme{\PP^\n}{\p}$ be an open
  embedding such that 
    the quotients   $ S_{A,R}/I_{A,R}$
   and $ S_{A,R+1}/I_{A, R+1}$ are free of rank $p(R)$ and $p(R+1)$ respectively,  and consider the restriction of $\chi_{|_U} \colon \chi^{-1}(U) \fd U $. 
    We check the surjectivity  of $\chi_{|_U}$.  

    Let us  choose isomorphisms $\tilde\rho:S_{A,R}/I_{A,R}\fd
  A^{\p(R)}$ and $\tilde\beta: 
  S_{A,R+1}/I_{A,R+1} \fd A^{\p(R+1)}$ and   their lifts  $\rho$ and $\beta$ to  $S_{A,R}$ and $S_{A,R+1}$.  We obtain a diagram as in
\eqref{diagram4} provided we can define a vertical map $\Sigma M$ 
that fulfills all the commutativity conditions required. 
  
 For this we let   $  M_i:=\tilde{\beta}\circ  \overline{\mu}_{i,R} \circ \tilde\rho^{-1} \colon A^{\p(R)}\rightarrow A^{\p(R+1)}$, for  $i=0, \dots, n$. 
By the commutativity of the second square and that of the multiplication by two variables $x_i,x_j$ we get
$$ M_{j}\circ \rho\circ  \mu_{i }=\tilde\beta\circ \pi_{R+1}\circ \mu_{j,R}\circ \mu_i=\tilde\beta\circ \pi_{R+1}\circ \mu_{i,R}\circ \mu_j =M_{i}\circ \rho\circ  \mu_{j}.$$

Therefore,  $(\rho, M)\in
\underline{C^p}(A)$, i.e. it defines a map $\alpha \colon \spec{A} \fd  C^p $ such that   $\chi\circ \alpha=\phi$, as   directly follows by the definition of $\chi$.

\medskip

As a consequence of the above, we observe that   a pair $(\rho, M)\in
\underline{C^p}(A)$,  is completely determined by the pair $(\rho,
\beta)$ of \eqref{eq:3.2}, 
since  $M$ is given by $  M_i:=\tilde{\beta}\circ  \overline{\mu}_{i,R}\circ \tilde{\rho}^{-1}$ 
for $i=0, \dots, n$. Therefore, in the sequel of this proof we will denote an element of $\underline C^p(A)$ as a triple $(\rho, M,\beta)$ and by $\widetilde{\rho}$  
and $\widetilde{\beta}$ the corresponding isomorphisms as in diagram \eqref{diagram4}.

\medskip

{\bf Claim 2}   $\chi:C^p\fd \hilbScheme{\PP^\n}{p}$ is a principal bundle  with  fibers  isomorphic to  $\gl{\p(R)}\times\gl{\p(R+1)}$. 

\medskip
 
The claimed property  is local on the Hilbert scheme.  Then, we may consider any open subset  $U=\spec{A}$ of $\hilbScheme{\PP^\n}{p}$  as above  and prove that  
 $\chi^{-1}(\spec{A})$ is isomorphic 
to $\spec{A}\times \gl{\p(R)}(A)\times \gl{\p(R+1)}(A)$.  In this case, as proved in the
Claim 1,  $(\rho, M,\beta ) \in \chi^{-1}(\spec{A})$
 if and only if  $\Ker(\rho)=I_{A,R}$, $ \Ker(\beta)=I_{A,R+1}$.  

We choose and fix an element  $(\rho_{\ast}, M_\ast, \beta_\ast)\in \chi^{-1}(\spec{A})$.

Then we can associate  to every   $(\rho, M,\beta )\in \chi^{-1}(\spec{A}$  the pair  $(\tilde\rho \circ \widetilde{\rho_{\ast}}^{-1}, \tilde\beta \circ
 \widetilde{\beta_{\ast}}^{-1}  )$ in $\gl{\p(R)}(A)\times\gl{\p(R+1)}(A)$. 

  On the other hand, we  can associate to any  pair of isomorphisms  $(\gamma_R, \gamma_{R+1})\in \gl{\p(R)}(A)\times\gl{\p(R+1)}(A)$, the triple 
$(\rho,  M, \beta)$ given by $\rho:=\gamma_R\circ \rho_\ast$, $\beta:=\gamma_{R+1} \circ \beta_\ast$ and $M_i:= \gamma_{R+1}\circ M_{\ast,i}\circ \gamma_R^{-1}$
 for every $i=0, \dots, n$.
  This   is again an element of  $\chi^{-1}(\spec{A})$ and, obviously, for different pairs $(\gamma_R, \gamma_{R+1})$ we obtain different elements 
$(\rho,  M, \beta)$ of $ \chi^{-1}(\spec{A})$.
\end{proof}

\section{Plucker coordinates}
\label{sec:Plucker coordinates}

Recall that there are two conventions for the Pl\"ucker
coordinates, which give different signs 
\cite[eq. 1.6]{gelfand-kapranov-zelevinsky94:multidemensionalDeterminants}.
The next propositions recall the basics about Grassmannians. They 
introduce the notations that we need and 
they precise our sign convention for the Pl\"ucker coordinates.

We denote by
$\GrassScheme{V}{q}$ the  Grassmannian   of codimension $q$ subspaces of a vector space 
$V$.  If  $E=\{e_1,\dots,e_v\}$ is an ordered  basis of  $V$,  for every $\kk$-algebra $A$, we also denote by $E$ 
the corresponding  basis $\{e_i\otimes 1_A\}$ of the free $A$-module $ V_A:=V\otimes_\kk A$.

\medskip

 A morphism $\spec A \rightarrow
  \GrassScheme{V}{q}$ is functorially defined by an inclusion of
  $A$-modules $W_A \subset V_A$ such
  that the quotient $V_A/W_A$ is locally free of rank $q$. The Pl\"ucker coordinates of a morphism $ f   \in \Hom(\spec A,
 \GrassScheme{V}{q})$ are defined  as follows. 

\begin{proposition}\label{Prop:plucker} Let  \noindent $F \colon \spec A \rightarrow \GrassScheme{V}{q}$ such that 
 $V_A/W_A$ is free of rank $q$ with basis
  \noindent $F$. Let $N  \in M_{q,v}(A)$ be the matrix  with columns $N_1, \dots N_v$
 corresponding to the canonical morphism 
$V_A \rightarrow V_A/W_A$  with respect to the bases
$E$ and \noindent $F$.    Consider a multi-index $(i_1,\dots,i_q)$ with 
$1\leq i_1 <i_2 < \dots  < i_{q}\leq  v $. The Pl\"ucker  coordinate  $P_{i_1\dots i_q}$  of  \noindent $F$ 
 is by definition the
determinant $det(N_{i_1},\dots,N_{i_q}) \in A$. It is well defined up to
multiplication by an invertible constant depending on the basis
 \noindent $F$.  Equivalently, it is 
    $(e_{i_1}\otimes 1_A)\wedge \dots
    \wedge (e_{i_q}\otimes 1_A) \in \Lambda^q (V_A/W_A)\simeq A$. 
\end{proposition}

 Let us consider for every multi-index  $(i_1,\dots,i_q)$ with $1\leq i_1<i_2  \dots <
i_q \leq v$ an indeterminate $X_{i_1,\dots,i_q}$  and the projective space  $\mathbb{P}=\Proj ( \kk[X_{i_1,\dots,i_q}])$ 
of dimension $\binom{v }{ q} -1$.  The Pl\"ucker embedding we now define  is compatible
with our convention for the Pl\"ucker coordinates. 
\begin{definition}
  The Pl\"ucker embedding $P:\GrassScheme{V}{q} \rightarrow \mathbb{P}=\Proj (  \kk[X_{i_1,\dots,i_q}])$ is the embedding characterized by the
  following: if \noindent $F\in \Hom( \spec A, \GrassScheme{V}{q}) $ is such that
  $V_A/W_A$ is free of rank $q$,  then $P\circ F\in \Hom( \spec
  A,  \mathbb{P})$ is described in coordinates by
  $X_{i_1,\dots,i_q}=P_{i_1,\dots,i_q}$. 
\end{definition}

\begin{remark}\label{rk:Pluckerordered}
 Starting from  \noindent $F:\spec{A}\fd C^p$, we define 
 \noindent $F_R \colon \spec A \rightarrow \GrassScheme{S_{R}}{\p(R)}$
 and \noindent $F_{R+1} \colon  \spec A \rightarrow
 \GrassScheme{S_{ R+1 }}{\p(R+1)}$ by the following compositions:
 \begin{eqnarray}\label{fR}
F_R:   \spec A \fd C^p \fd \hilbScheme{\PP^\n}{\p} \fd
   \GrassScheme{S_{R}}{\p(R)}\\ \label{fR+1}
   F_{R+1}:\spec A \fd C^p \fd \hilbScheme{\PP^\n}{\p} \fd
   \GrassScheme{S_{R+1}}{\p(R+1)}.
 \end{eqnarray}
 We  associate to each  tuple $\underline z=(z_1,\dots,z_{\p(s)}) $  of  monomials  of degree $s$ (where $s=R$ or  $s=R+1$)  ordered  increasingly, a   Pl\"ucker coordinate  on $\GrassScheme{S_{s}}{\p(s)} $ that we  denote 
$P_{z_1,\dots,z_{\p(s)}}$ or $P_{\underline z}$.
To simplify statements and proofs, we also consider  tuples  of monomials   $\underline z $  possibly not ordered in increasing order and with  possibly 
repeated monomials, and associate to them the symbol $P_{\underline z}$ (that, by abuse, we call \Pl\ coordinate).    If two lists of monomials $\underline z$ 
and $\underline z'$  are equal up to a permutation, then $P_{\underline z}=\pm P_{\underline z'}$ with the sign given by the parity of the permutation. 
Then,  if    $\underline z$  contains  repeated  monomials,   $P_{\underline z}$ simply stands for 0, while if the monomials are all distinct, 
$ P_{\underline z}$ is a  true \Pl\ coordinate  up to a sign.
\end{remark}

The next proposition 
describe the Pl\"ucker coordinates of 
 \noindent $F_R$ and \noindent $F_{R+1}$
in terms of  the entries 
of the matrices $\rho,M_0,\dots,M_n$ which are associated to \noindent $F$
through the functorial description of $C^p$.

\begin{proposition} \label{Plcoordinates+1}In the above notations, 
  the Pl\"ucker coordinates $P_{z_1,\dots,z_{\p(R)}}$  of \noindent $F_R$ are the
  maximal minors of $\ro$. The Pl\"ucker coordinates
  $P_{v_1,\dots,v_{\p(R+1)}}$ of \noindent $F_{R+1}$  
  are the maximal minors of
 $\Sigma  M\circ  (\rho \oplus \dots \oplus \rho )$. More precisely, if 
for each monomial $v_i\in S_{R+1} $, we choose   a monomial
$z_{t(i)}\in S_{R} $ and  
a variable   $x_{j(i)}$  such that $v_i=x_{j(i)} z_{t(i)}$ and set
$\tilde z_i=(0,\dots,0,z_{t(i)},0,\dots,0)  \in (
S_{A,R})^{\n+1}   $, where $z_{t(i)}  \in S_{A,R}  $ is located at position $j(i)$ so  that $\Sigma\mu(\tilde
  z_i)=\mu_{  j(i)  }(z_{t(i)})=
v_i$,   then $P_{v_1,\dots,v_{\p(R+1)}}$ is the determinant of
    the matrix whose  $i$-th column is  $C_i:= \Sigma M\circ (\rho \oplus \dots \oplus \rho)(\tilde z_i)$.
\end{proposition}
\begin{proof}
From our constructions, we have the two following diagrams with exact
rows and commutative squares. 
 \begin{displaymath}
    \begin{array}{cccccccccc}
 0&\fd &  I_{A,R} &\fd  &  S_{A,R} &
&\stackrel{ \rho}{\fd} &
A^{\p(R)} &\fd & 0 
    \end{array}
 \end{displaymath}

  \begin{displaymath}
    \begin{array}{cccccccccc}
0&\fd &  (I_{A,R})^{\n+1} &\fd  & ( S_{A,R})^{\n+1} &
&\stackrel{\rho \oplus \dots \oplus \rho }{\fd} &
(A^{\p(R)})^{\n+1} &\fd & 0 \\
    &&  \downarrow \Sigma\mhu{I,R} & & \downarrow \Sigma\mhu{R}                   && &
      \downarrow \Sigma M && \\
0 &\fd & I_{A,R+1} & \rightarrow &  S_{A,R+1} & &\stackrel{\beta}{ \rightarrow} &
 A^{\p(R+1)}& \rightarrow & 0 \\      
    \end{array}
 \end{displaymath}
Using the functorial description of the Grassmannian, the morphism
\noindent $F_R$ is described by the inclusion  $I_{A,R}
\subset  S_{A,R}$.  The first line shows
that the Pl\"ucker coordinates in degree $R$ are given by the
maximal minors of $\rho$.

The morphism
\noindent $F_{R+1}$ is described by the inclusion  $I_{A,R+1}
\subset  S_{A,R+1}$. The last line  shows
that the Pl\"ucker coordinates in degree $R+1$ are given by the
maximal minors of $\beta$.  
Since 
$\Sigma\mu_R  $ is surjective and sends the monomial basis of $(
S_{A,R})^{\n+1} $
to the monomial basis of $S_{A,R+1}$, the maximal minors of 
$\beta$ coincide with the maximal minors of
$\beta\circ \Sigma \mu_R=\Sigma M\circ  \rho \oplus \dots \oplus \rho $.  More precisely,
if for each monomial $v_i\in S_{A,R+1}$,  we choose a monomial  $\tilde z_i$, as described in the statement, then $\beta(v_i)=
\beta(\Sigma \mu_R(\tilde z_i))=(\Sigma M\circ  \rho \oplus \dots \oplus \rho )(\tilde
z_i)$. The Pl\"ucker coordinate $P_{v_1,\dots,v_{\p(R+1)}}$ is the determinant built with the $\beta(v_i)$ as 
columns, so the second equality of the proposition follows. 
\end{proof}

\medskip

From the description of the Pl\"ucker coordinates, 
we get for free the vanishing of some Pl\"ucker
coordinates over the whole Hilbert scheme   if  the  Hilbert polynomial $p$ has a positive degree.   Indeed, among the minors of  $\Sigma  M\circ  
(\rho \oplus \dots \oplus \rho )$ are the minors of $M_i\circ
\rho$, and they vanish for a  nonconstant $p$. 
The idea is as follows, in the case where 
the image $Im(M_i\circ \rho)$  is a free module.
In the composition $  S_{A,R} 
\stackrel{ \rho}{\fd} 
A^{\p(R)}  \stackrel{ M_i}{\fd} A^{\p(R+1)} $,  the rank
of the image $Im(M_i\circ \rho)$ is at most $\p(R)$,  the rank  
of the space in the middle, so that all minors of $M_i\circ \rho$ of order
 $\p(R+1)>p(R)\geq rank(M_i\circ \rho)$  vanish. 
Beyond the vanishing of these particular Pl\"ucker coordinates, 
it is possible to get more general linear equations using 
a similar trick and elaborating on the above observation.

\section{A stratification of the Grassmannian $ \GrassScheme{S_{R+1}}{\p(R+1)}$}\label{sec:strat}

In this section we introduce for every integer $b\geq \p(R)$ a
subscheme $\mathbf{H}^{(b)}\subset
\GrassScheme{S_{R+1}}{\p(R+1)}$ of  the Grassmannian whose closed
points are the vector spaces $I_{R+1}\subset S_{R+1}$ with
$codim((I_{R+1}:l)\subset S_R)<b$. 
We will prove that $\mathbf{H}^{(b)}$  is cut out by a linear space,  is empty for $b=\p(R)$ and contains the Hilbert scheme for  $b>\p(R)$.  

 We recall that  a map $\spec{A}\rightarrow {\GrassScheme{S_{R+1}}{\p(R+1)}}$  is given in the functorial description of the Grassmannian by a submodule  
$I_{A,R+1}\subset S_{A,R+1}$ with locally free quotient of rank $p(R+1)$.

\begin{definition}\label{def:functorH}
  Let $X$ be a noetherian scheme. 
  For every  integer $b\geq \p(R)$, we say that a morphism $X\to
  \GrassScheme{S_{R+1}}{\p(R+1)}$ satisfies the property 
  ${\mathcal{P}}^{b}$ if  for 
  every   noetherian $k$-algebra $A$ and every morphism $\spec A \to X$,
  the composed map  $\spec{A}\rightarrow {\GrassScheme{S_{R+1}}{\p(R+1)}}$  satisfies  $\wedge^{b}S_{A,R}/(I_{A,R+1}\colon \ell)=0 $ for every $\ell \in S_{1}$.
\end{definition}

It is not obvious that this definition is local as 
in general the colon ideal   does not commute with the change of
scalars, namely  for a morphism of $k$-algebras $A\rightarrow B$ the
module $S_{A,R}/(I_{A,R+1}:\ell)\otimes_AB$ can be different from $
S_{B,R}/(I_{A,R+1}\otimes_A B :\ell)$. 
However, this locality is true, as formulated in the next
proposition. The proof is an easy application of  Lemma \ref{prop:contiene}.  

\begin{proposition}\label{prop:contieneExplicit}     
  The morphism $X\to
  \GrassScheme{S_{R+1}}{\p(R+1)}$ satisfies the property 
  ${\mathcal{P}}^{b}$ if  for some (or any) open covering of $X$ by
  affine subschemes  $\spec{A_i} 
  \to X$, the equality $\wedge^{b}S_{A_i,R}/(I_{A_i,R+1}\colon \ell)=0
  $ holds for all $i$ and $l\in S_1$.  
\end{proposition}

\begin{lemma}\label{prop:contiene}     
If  $\spec{A}\rightarrow {\GrassScheme{S_{R+1}}{\p(R+1)}}$  is  given  by the submodule  $I_{A,R+1}\subset S_{A,R+1}$, then 
       \begin{enumerate}
       \item $\wedge^{b}S_{A,R}/(I_{A,R+1}\colon \ell)=0$  if and only if    $\wedge^{b}S_{A_{\mathfrak p},R}/(I_{A_{\mathfrak p},R+1}\colon \ell)=0$    
for every prime (or maximal)  ideal $\mathfrak p$ of $A$.     

\item If    $A\fd B$ is a morphism  of  $\kk$-algebras  and $\wedge^{b}S_{A,R}/(I_{A,R+1}\colon \ell)=0$, then
also  $\wedge^{b}S_{B,R}/(I_{B,R+1}\colon \ell)=0$.  
\end{enumerate}
\end{lemma}
\begin{proof}
 $(1)$ Let  us consider  any prime (or maximal)  ideal $\mathfrak p$ of $A$  and any linear form $\ell$. 
 
The localization commutes with  the colon ideal over a finitely generated ideal \cite[Corollary 3.15]{atiyahMacdonal:introductionToCommutativeAlgebra}. 
 Then,  we  have the equalities 
$$S_{A_{\mathfrak p},R}/(I_{A_{\mathfrak p},R+1}\colon \ell)=S_{ A_{\mathfrak p},R} /((I_{A,R+1}\colon \ell)\otimes_A A_{\mathfrak p})=
(S_{A,R} \otimes_A A_{\mathfrak p} )/((I_{A,R+1}\colon \ell)\otimes_A A_{\mathfrak p}). $$
 Moreover, the localization commutes with   the quotient \cite[Corollary 3.4]{atiyahMacdonal:introductionToCommutativeAlgebra} and with 
the  exterior powers   \cite[Proposition A2.2 b]{eisenbud95:commutativeAlgebraWithAView}, so that 
$$\wedge^{b}(S_{A_{\mathfrak p},R}/(I_{A_{\mathfrak p},R+1}\colon \ell))=\wedge^{b}((S_{A,R}  /(I_{A,R+1}\colon \ell))\otimes_A A_{\mathfrak p} )=
 (\wedge^{b}S_{A,R}/(I_{A,R+1}\colon \ell) )\otimes_A A_{\mathfrak p}   .$$
Therefore, if $\wedge^{b}S_{A,R}/(I_{A,R+1}\colon \ell) =0$, then also $\wedge^{b}S_{A_{\mathfrak p},R}/(I_{A_{\mathfrak p},R+1}\colon \ell)=0$
 for every prime (or every maximal) ideal $\mathfrak p$ of $A$. Moreover, also the converse is true, since 
 the property of being the null module  is  local  \cite[Proposition  3.8]{atiyahMacdonal:introductionToCommutativeAlgebra}.

\medskip
  
  $(2)$  Recall that by the functorial definition of Grassmannian,   $I_{B,R+1}=I_{A,R+1}\otimes_AB$.
 Tensorizing  the sequence 
\begin{equation} \label{standard}0\fd(  I_{A,R+1}\colon \ell)\fd  S_{A,R} \fd S_{A,R}/(  I_{A,R+1}\colon \ell)\fd 0\end{equation}
we get 
\begin{equation} \label{tensorB}  (  I_{A,R+1}\colon \ell)\otimes_AB \stackrel{f}\fd  S_{B,R} \fd (S_{A,R}/(  I_{A,R+1}\colon \ell))\otimes_A B\fd 0.\end{equation}

We observe that \noindent $f( (  I_{A,R+1}\colon \ell)\otimes_AB )$ is contained in     $(I_{B,R+1}\colon \ell)$; indeed,             if $m $ is any element in $S_{A,R}$ 
such that $\ell m \in I_{A,R+1}$, then $m\otimes_A 1_B$ belongs to  $(I_{B,R+1}\colon \ell)$  since  $\ell (m\otimes_A 1_B)=\ell m \otimes_A 1_B$
 belongs to $ I_{A,R+1} \otimes_A B$.
 Therefore, there is a surjective  map $ S_{B,R}/f((I_{A,R+1}\colon \ell)\otimes_A B) \fd S_{B,R}/(I_{B,R+1}\colon \ell)$; by  
 \cite[Prop. A.2.2,d]{eisenbud95:commutativeAlgebraWithAView}  also the following is surjective
\begin{equation}\label{surjec}  \wedge^{b}S_{B,R}/f((I_{A,R+1}\colon \ell)\otimes_A B) \fd \wedge^{b}S_{B,R}/((I_{B,R+1}\colon \ell)\fd 0.  \end{equation}
By \eqref{tensorB},  we see that  $\mathrm{coker}(f) =S_{B,R}/ f(( I_{A,R+1}\colon \ell)\otimes_A B)\simeq (S_{A,R}/(  I_{A,R+1}\colon \ell))\otimes_A B $.
Then, applying \cite[Proposition A.2.2,b]{eisenbud95:commutativeAlgebraWithAView}  we get  $\wedge^{b} \left( (S_{A,R}/(  I_{A,R+1}
\colon \ell))\otimes_A B\right) =\left(\wedge^{b}(S_{A,R}/ ( I_{A,R+1}\colon \ell)\right)  \otimes_A B=0 $.    We conclude  by \eqref{surjec}.
\end{proof}

\begin{remark}  \label{rk:isfunctReformulated}
  We can
  easily 
  see that  a morphism $\spec{A}\rightarrow
  {\GrassScheme{S_{R+1}}{\p(R+1)}}$  which satisfies  $ {\mathcal  P}^b$, also  satisfies  $ {\mathcal  P}^{b'}$
for $b<b'$.    Moreover, 
all the morphisms $\spec{A}\rightarrow
  {\GrassScheme{S_{R+1}}{\p(R+1)}}$ satisfy the property  $ {\mathcal
    P}^{p(R+1)+1}$, as can be proved by a variation on proposition  \ref{prop:caracterizationH}.


Therefore, in the following we will study only the properties   $ {\mathcal  P}^b$ with $b\leq p(R+1)$.

Furthermore, for a constant  Hilbert polynomial 
$p(t)=d$, then $p(R)=p(R+1)=d$ so that we may reduce to study $ {\mathcal  P}^b$ with $b\leq d=p(R)$.

\end{remark}

\begin{proposition}\label{prop:caracterizationH}
Let $b\leq p(R+1)$. The following are equivalent for a map $\alpha \colon \spec{A}\fd   {\GrassScheme{S_{R+1}}{\p(R+1)}}$
\begin{enumerate}[i)]
\item   $\alpha$ satisfies ${\mathcal P}^{b}$;
 \item for every tuples  $\underline z$  of $b$ monomials    in $S_{R}$ and   $\underline v$ of ${\p(R+1)}-b$ monomials   in $S_{R+1}$, the image of 
$e^{(b)}(\ell,\underline z,\underline v):=\wedge(\ell \underline  z,   \underline v)$
  in $\wedge^{\p(R+1)}(S_{A,R+1}/I_{A,R+1})$   vanishes for every $\ell \in S_{A,1}$ (or  
 for every $\ell$  in a dense subset of $S_1$,   or for every
 $\ell$ in $S_1$  ).
 \end{enumerate}
\end{proposition}
\begin{proof}


We first prove the equivalence between $i)$ and the alternative in
$ii)$  for all $l \in S_1$.

By proposition \ref{prop:contieneExplicit}, 
the condition $i)$ asserts that $ \wedge^{ b}
(S_{A,R}/(I_{A,R+1}:l))=0$ or equivalently that 
\begin{eqnarray*}
  \wedge^{ b}  S_{A,R}&\fd& \wedge^{ b} (S_{A,R}/(I_{A,R+1}:l))
\end{eqnarray*}
is the null morphism for every $l$. This is also equivalent to the
vanishing of the morphism
\begin{eqnarray*}
  \wedge^{ b}  S_{A,R}&\stackrel{\phi_l}{\fd}& \wedge^{ b}
                                               (S_{A,R+1}/I_{A,R+1})\\
z_1\wedge \dots \wedge z_{b}&\mapsto& lz_1\wedge \dots
\wedge lz_{b}
\end{eqnarray*}
since both morphisms have the same kernel
 $(I_{A,R+1}:l)\wedge
\wedge^{ b-1}  S_{A,R}$
according to
\cite[Prop. A.2.2,d]{eisenbud95:commutativeAlgebraWithAView}.
The vanishing of $\phi_l$ is equivalent to the vanishing of
$\wedge(\ell \underline  z)$
  in $\wedge^{b}(S_{A,R+1}/I_{A,R+1})$ for all monomials $z_i$, or to
  the vanishing of  
$e^{(b)}(\ell,\underline z,\underline v):=\wedge(\ell \underline  z,
\underline v)$ for all monomials $z_i,v_i$.

\medskip

Finally, we prove that the  alternatives in $ii)$ are equivalent, namely that   the vanishing of $e^{(b)}(\ell,\underline z,\underline v)$ 
for $\ell$ in a dense subset of $S_1$ implies its vanishing for all $\ell\in S_{A,1}$.

    Let us consider the linear form \noindent $L:=y_0x_0+\dots+y_nx_n\in S_{B,1}$,  where $B$ is the polynomial ring   $A[y_0,\dots, y_n]$
 in the indeterminates $y_0,\dots, y_n$, and let  $\underline z$   and $\underline v$ as in the statement.
 If we formally develop  $ \wedge (L \underline z,  v)$  with respect to the indeterminates  $y_0, \dots, y_n$ and   coefficients 
 in  $\wedge^{\p(R+1)}S_{R+1}$,   we obtain  a homogeneous polynomial  of degree $b$.  
If  we now consider the image of the coefficients under the projection  
 $\wedge^{\p(R+1)}S_{R+1} \fd \wedge^{\p(R+1)}S_{A,R+1}/I_{A,R+1}\simeq A$,   we obtain  polynomials    
$Q^{(b)}_{\underline z, \underline v}$ in $A[y_0, \dots, y_n]$. 
 When \noindent $L$ specializes to any $ \ell \in S_ {1}$, the polynomial $Q^{(b)}_{\underline z, \underline v}$  specializes 
   to  $e^{(b)}(\ell,\underline z,\underline v)$. We then conclude by Lemma \ref{ZariskiDense}  \eqref{ZariskiDense1}. 
\end{proof}

\begin{definition}\label{not:defH}
For every $b\leq \p(R+1)$, let $\underline z=(z_1,\dots,z_{b})$ be a tuple of monomials in
$S_{R}$,  $\underline v=(v_{b+1},\dots,v_{\p(R+1) })$ be a tuple of 
monomials in
$S_{R+1}$, and  $\underline x=(x_{i_1},x_{i_2}..., x_{i_{b}})$
be a tuple of   variables, all of them ordered  increasingly.  We will denote by $ P_{\underline x , \underline z, \underline v}$ be
the Pl\"ucker
coordinate on $\GrassScheme{S_{R+1}}{\p(R+1)}$ associated to
  the tuple 
$( x_{i_1} z_1,...., x_{i_b}z_{b},v_{b+1},\dots,v_{\p(R+1) })$ 
and  by   $\mathbf H_{\underline x,\underline z,\underline v}^{(b)} 
 $  
 the hyperplane section of $ \GrassScheme{S_{R+1}}{\p(R+1)} $  given  by the vanishing of the  linear form $E^{(b)}(\underline x,\underline
z,\underline v):=\sum P_{\sigma(\underline x),\underline z, \underline v }$
where the sum runs over all the possible distinct permutations $\sigma(\underline x)$ of the tuple  $\underline x$  
(two permutations  $\sigma(\underline x)$  and $\sigma'(\underline x)$
are distinct if $\sigma(\underline x)$  and $\sigma'(\underline x)$
are different as ordered tuples).

The scheme ${\mathbf  H}^{(b)}$ is by definition the  closed \ subscheme   
$\GrassScheme{S_{R+1}}{p(R+1)}$ cut out by the hyperplanes $\mathbf{H}^{(b)}_{\underline x,\underline z,\underline v}$ 
for every  tuples  $\underline x$, $\underline z$   and   $\underline v$.
\end{definition}

\begin{proposition} \label{fondamentale7}
  A morphism $\alpha:X\fd \GrassScheme{S_{R+1}}{\p(R+1)}$ satisfies
  the property ${\mathcal P}^b$ if and only if it factorizes through
  ${\mathbf  H}^{(b)}$.
\end{proposition}

\begin{proof}  We can check the statement locally, namely considering a local $k$-algebra $A$ and a map
 $\alpha \colon \spec{A}\fd \GrassScheme{S_{R+1}}{\p(R+1)}$, so that
 $S_{A,R+1}/I_{A,R+1}$ is free and $\alpha$   has \Pl\ coordinates. 
We exploit the argument presented in the proof of Proposition \ref{prop:caracterizationH}.
The   coefficient in  $Q_{\underline z, \underline v}^{(b)}\in A[y_0, \dots, y_n]$ of each monomial $y_{i_1}\cdots y_{i_b}$ is (up to a sign) 
 the value at  $I_{A,R+1}$ by the linear form $ E^{(b)}(\underline x,\underline
z,\underline v )$ with $\underline x=(x_{i_1},\dots , x_{i_b})$. Thus
$\alpha$ factorizes through ${\mathbf  H}^{(b)}$ iff  all the polynomials
$Q_{\underline z, \underline v}^{(b)}$ vanish. 
By   Proposition  \ref{prop:caracterizationH},  this is
equivalent to the property
$ {\mathcal P}^{b} $  for $\alpha$.
\end{proof}

\begin{remark} \label{rk:isfunct} Reformulating Remark
  \ref{rk:isfunctReformulated} with Proposition \ref{fondamentale7},
we get that   $ {\mathbf  H}^{(b)}$  is a subscheme  of $ {\mathbf
  H}^{(b')}$ if $b<b'$, and    $ {\mathbf{H}}^{(b)}=
{\GrassScheme{S_{R+1}}{\p(R+1)}}$  for every   $b\geq \p(R+1)+1$.
Therefore, in the following we will study only the schemes $ {\mathbf{H}}^{(b)}$ with $b\leq p(R+1)$.
\end{remark}

\section{The schemes $\mathbf E$,  $\mathbf F$ and $\mathbf \Delta$}
\label{sec:6}

\begin{notation}    \label{rk:uguale}
  To bridge the notations of the previous section and the notations of
  of Theorem B {\it 2)} and {\it 3)}, we observe that
  by definition  $ E(\underline x,\underline
z,\underline v )   = E^{(\p(R)+1)}(\underline x,\underline
z,\underline v)$  and   $F(\underline x,\underline
z,\underline v )=E^{(\p(R))}(\underline x,\underline
z,\underline v)$.

We will denote by  ${\mathbf E}$ and  ${\mathbf F}$  the closed subschemes  of $\GrassScheme{S_{R+1}}{\p(R+1)}$
that are defined  by the  linear forms   $E(\underline x,\underline z,\underline v)$ and  
\noindent $F(\underline x, \underline z,\underline v )$.
In other words, 
${\mathbf E}= {\mathbf  H}^{(\p(R)+1)}$  and ${\mathbf F}= {\mathbf  H}^{(\p(R))}$.

We will denote by $\mathbf \Delta$ the closed subscheme of of $\GrassScheme{S_{R+1}}{\p(R+1)}$ defined by the quadratic  equations  $F(\underline x,\underline {z},\underline {v} )F(\underline x', \underline {z'},\underline {v'})-
 F(\underline x,\underline {z'},\underline {v})F(\underline x',\underline {z},\underline {v'})$ of Theorem B {\it 3)}.
\end{notation}

The present section is devoted to a further study of the schemes  $\mathbf E$,  $\mathbf F$ and $\mathbf \Delta$.
First we  give an intrinsic  description of ${\mathbf E}$ that depends on, and in some sense generalizes,   Green's  Theorem  (Theorem \ref{macaulayMaximalGrowth} \eqref{macaulayMaximalGrowth_3}). In fact we can rephrase the following result saying that $\mathbf E$ is the locus of the Grassmannian where Green's bound is sharp.

 \begin{theorem} \label{prop:generalform1} Let  us consider a map  $\alpha \colon \spec{A}\fd \GrassScheme{S_{R+1}}{\p(R+1)}$ with  
$(A, \mathfrak m,K)$   a local   $\kk$-algebra.  The following are equivalent 
 \begin{enumerate}[1)]
\item \label{prop:generalform11}
 $\alpha$ factorizes through ${\mathbf E}$
\item \label{prop:generalform12} for every  general  $\ell \in S_1$ the quotient  $J_\ell:=S_{A,R}/(I_{A,R+1}\colon \ell )$ is  free  of rank $ p(R) $,  
namely  it  corresponds to a map $\alpha_{R,\ell} \colon \spec{A}\fd \GrassScheme{S_{R}}{p(R)}$.
\end{enumerate}
\end{theorem}

\begin{proof}
2)  $\Rightarrow$ 1) . In fact for every tuple $\underline z$ of $\p(R)+1$  monomials 
in $S_R$  and a general $\ell$ in $S_1$,     $\wedge  \underline z$  vanishes  since it is an element of  $\wedge^{\p(R)+1}S_{A,R}/(I_{A,R+1}\colon \ell )=0$. 
 Hence also $\wedge \ell \underline z, \underline v$ vanishes in
 $\wedge^{p(R+1)}(S_{A,R+1}/I_{A,R+1})$. Thus $\alpha$ has property
 $\mathcal P^{p(R+1)}$ from Proposition  \ref{prop:caracterizationH}
 and we conclude by Proposition \ref{fondamentale7}. 

\medskip

1)  $\Rightarrow$ 2)
 As $A$ is local,  $S_{A,R+1}/I_{A,R+1}$ is free with rank $p(R+1)$ and  $I_{A,R+1}$ is free with rank $q(R+1)$.  Recall that $N(t)$ and $q(t)$ are 
$\dim_\kk S_t$ and respectively $N(t)-\p(t)$.
 
\smallskip

 We choose any $\ell$ (general) such  that Green's Theorem holds for
 $I_{K,R+1}=I_{A,R+1}\otimes_A K$.   Green's results holds for
 every $\ell$ in a suitable open subset $U$ of $K^n$. According to
 lemma \ref{ZariskiDense},
 we may choose $\ell$ general in $k^n=S_1$.
 We then perform a change 
 of coordinates leading $\ell$ to $x_0$.



In this way  we  have    $d:=\dim_K S_{K,R}/(I_{K,R+1}:x_0)\geq p(R)$.   We can then choose a tuples $\underline z$ 
of $d$ monomials in $S_R$ and $\underline v$ of $\p(R+1)-d$ monomials in $S_{R+1}$ such that $\underline z$ is a basis of $S_{K,R}/(I_{K,R+1}:x_0)$ and 
$x_0\underline z, \underline v$ is a basis of $S_{K,R+1}/I_{K,R+1}$. As $\wedge(x_0\underline z,\underline v)$ is non zero  when computed in  $\wedge^{\p(R+1)}S_{K,R+1}/I_{K,R+1}=\wedge^{\p(R+1)}S_{A,R+1}/I_{A,R+1}\otimes_A A/\mathfrak m$, then it is also invertible  in 
$\wedge^{\p(R+1)}S_{A,R+1}/I_{A,R+1}$.

  \smallskip
  
  We now   construct a special basis for  $I_{A,R+1}$ starting from $x_0\underline z, \underline v$.
   
 Let $B=\{x_0\underline w, \underline u\}$    be the set of $q(R+1)$ monomials in $S_{R+1} \setminus \{x_0\underline  z, \underline v\}$, where 
$\underline u$ is the tuple of those not divisible by  the variable $x_0$: note that by construction  $\underline w, \underline z$ is the complete 
list of monomials in $S_R$; hence $\underline w$  contains    $c:=N(R)-d\leq q(R)$ monomials.
By Nakayama, every monomial in $B$ can be written modulo $I_{A,R+1}$ as a linear combination of monomials in $\{x_0\underline z, \underline v\}$, thus
we can find in $I_{A,R+1}$ polynomials    \noindent $T_1, \dots, T_{q(R+1)}$  such that the $i$-th monomial in  $B$ appears only in \noindent $T_i$ and its coefficient is $1_k$. 

These polynomials $T_i$ are in fact a free set of generators  of
$I_{A,R+1}$. Indeed, every polynomial \noindent $G \in I_{A,R+1}$ has a unique
writing  \noindent $G=\sum_{i=1}^{q(R+1)} a_i T_i$, where each $a_i$
is the coefficient in  \noindent $G$ of the $i$-th monomial  of $B$.
Unicity is clear considering the coefficients on the monomials of
$B$. On the other hand $G-\sum_{i=1}^{q(R+1)} a_i T_i$ lies in
$I_{A,R+1}$, is a linear combination of the monomials $
\{x_0\underline  z, \underline v\}$ which form a base of
$S_{K,R+1}/I_{K,R+1}=S_{A,R+1}/I_{A,R+1}\otimes K$, hence by Nakayama
a base of $S_{A,R+1}/I_{A,R+1}$.

Hence, for  every $D\in  (I_{A,R+1} \colon x_0)$  we have       $x_0D =\sum_{i=1}^c a_i T_i$. 
   Furthermore,  if \noindent $T_i=x_0T_i'+T_i''$ with \noindent $T_i''\in A[x_1, \dots, x_n]$,  then $x_0D=x_0  \sum_{i=1}^c a_i T_i'$  since the summand $\sum_{i=1}^c a_i T_i''$
 belongs to  $x_0S_{A,R}\cap A[x_1, \dots, x_n]=\{0\}$. Therefore,  $D=\sum_{i=1}^c a_i T_i'$, so that  $(I_{A,R+1}\colon x_0)$ is contained in the
 $A$-submodule $Q$ of $S_{A,R}$ generated by the polynomials \noindent
 $T_1', \dots, T_c'$: note that by construction these polynomials are
 linearly independent on $A$ because their matrix of coefficients
 corresponding to the monomials of $\underline w$ is  the identity.  
This matrix property also shows that    $S_{A,R}=Q\oplus P$ with $P$ the free submodule of rank
$d=N(R)-c\geq \p(R)$ generated by the monomials in $S_R\setminus
\{\underline w\}$.  

Then, in the standard   exact sequence 
\begin{equation} \label{esatta} 0\fd  (I_{A,R+1}\colon x_0) \stackrel{i}\rightarrow   S_{A,R} =Q\oplus P \rightarrow S_{A,R}/ (I_{A,R+1}\colon x_0)\rightarrow 0\end{equation}
the image of the first map  is contained in  $  Q$.   Therefore, $ S_{A,R}/ (I_{A,R+1}\colon x_0) $ is isomorphic to $Q/\Imm(i) \oplus P$. By hypothesis,   $ \wedge^{\p(R)+1}S_{A,R}/ (I_{A,R+1}\colon x_0)=0 $, hence $\rk{P}=\p(R)$ and $Q/\Imm(i)=0$, namely  $ S_{A,R}/ (I_{A,R+1}\colon x_0) \simeq P$ is free of rank $\p(R)$ and $(I_{A,R+1}\colon x_0)=Q$ is free of rank $q(R)$. 
\end{proof} 

Now  we prove that $\mathbf F$ is empty and then use this fact to  give an intrinsic description of the scheme theoretical intersection  $\mathbf E\cap \mathbf  \Delta $.

\begin{lemma}\label{prop:nonvanishing1}  For every 
  map $ \spec{A}\fd \GrassScheme{S_{R+1}}{\p(R+1)}$  with $A$ local,  there exist tuples $ \underline z_0$ and  
$ \underline v_0$  of monomials such that the following elements are invertible at $ I_{A,R+1}$
  \begin{itemize}
  \item  $e^{(\p(R))}(\ell_0,\underline z_0,\underline v_0) =\wedge (\ell_0 \underline z_0, \underline v_0)$  for every general linear form  
$\ell_0$ in  $S_1$.
  \item   \noindent $F(\underline x,\underline z_0,\underline v_0)$  for  a suitable  tuple   $\underline x$ of $\p(R)$ variables.
  \end{itemize}
Therefore, the scheme $\mathbf{F}$ is empty. 
\end{lemma}
\begin{proof}
We first prove  the result assuming that $A=K$ is a field.  If  $b$ is the dimension of the $K$-vector space $S_{K,R}/(I_{K,R+1}:\ell)$ 
with  $\ell$     a general linear form in $S_{1,K}$, by Green's  Theorem \ref{macaulayMaximalGrowth} \eqref{macaulayMaximalGrowth_3} 
we have $b\geq p(R)$. By Lemma \ref{ZariskiDense} 
 this is also true for a general $\ell_0 \in S_1$. Therefore,
$\wedge^{p(R)}S_{K,R}/(I_{K,R+1}\colon \ell_0)\neq 0 $.
It follows by Proposition \ref{prop:caracterizationH} that
 $e^{(\p(R))}(\ell_0,\underline z_0,\underline v_0) =\wedge (\ell_0
 \underline z_0, \underline v_0)$ is
invertible in  $\wedge^{p(R+1)}S_{K,R+1}/I_{K,R+1}\simeq K$.

We can extend the proof to the case of a local ring $(A,\mathfrak m,K)$  applying the first part to  $ \spec{K}\fd \spec{A}\fd 
\GrassScheme{S_{R+1}}{\p(R+1)}$. If $e^{(\p(R))}(\ell_0,\underline z_0,\underline v_0)$  gives a non-zero element of  $\wedge^{\p(R+1)}S_{K,R+1}/I_{K,R+1}\simeq K$,  
then  it gives an invertible element of $\wedge^{p(R+1)}S_{A,R+1}/I_{A,R+1}\simeq A$, since this element does not vanish modulo $\mathfrak m$.

 The second item follows from the first one by Proposition \ref{fondamentale7}. 
As a consequence,  $\mathbf F$ is empty since it has no closed points.
\end{proof}

\begin{lemma}\label{lem:equiv}
Let  $A$ be a local $\kk$-algebra and  $\alpha \colon \spec{A}\fd \GrassScheme{S_{R+1}}{\p(R+1)}$. The following are equivalent:
\begin{enumerate}
\item \label{eq:degreetwo}  $\alpha$ factorizes through $\mathbf \Delta$ 
\item\label{eq:manywedge}    
$ \wedge (\ell \underline z,  \underline v)\   \cdot\  { \wedge (  \ell'\underline z',   \underline v') }  -   {\wedge ( \ell' \underline z,  \underline v')}\ \cdot
\ { \wedge (  \ell\underline z',  \underline v) }$  vanishes in    $\wedge^{\p(R+1)} S_{A,R+1}/I_{A,R+1}\simeq A$
for every   tuples $(  \underline v,\underline z)$ and $ ( \underline v',\underline z')$  and every (general)   linear form $\ell, \ell'$ in $S_1$;
\item $ \wedge (\ell \underline z,  \underline v)\   \cdot\  { \wedge (  \ell_0\underline z_0,   \underline v_0) }  -   {\wedge ( \ell_0 \underline z,  \underline v_0)}\ \cdot
\ { \wedge (  \ell\underline z_0,  \underline v) }$  vanishes in    $\wedge^{\p(R+1)} S_{A,R+1}/I_{A,R+1}\simeq A$
for every   tuples $(  \underline v,\underline z)$,  every general  $\ell$ in $S_1$ and 
 $(\ell_0,\ \underline z_0, \ \underline v_0)$ such that $\wedge (\ell_0 \underline z_0, \underline v_0)$ is invertible.
\end{enumerate}
\end{lemma}

\begin{proof}
  The equivalence of the the first two conditions  is obtained by a direct computation, similarly to Proposition \ref{fondamentale7}:   
  consider two linear forms \noindent $L=x_0y_0+\dots, x_n y_n$ and \noindent $L'=x_0y'_0+\dots, x_n y'_n$ with $y_i, y'_j$ indeterminates and  formally expand 
$M(y_i, y'_i, \underline z,  \underline v,\underline z',  \underline v'):= \label{eq:degreetwoL}{\wedge ( L \underline z,  \underline v)}   \cdot  { \wedge ( L'\underline z',  \underline v') }   -  
  {\wedge (  L' \underline z,    \underline v')} \cdot  {\wedge ( L  \underline z',  \underline v) }$
with respect to the indeterminates $y_i, y'_j$. 

By definition of $\mathbf \Delta$, the condition $(\bcr 1 \ecr)$ is equivalent to the vanishing of 
all the coefficients  in these expansions, hence is equivalent to the
fact that all the  $M(y_i, y'_i, \underline z,  \underline
v,\underline z',  \underline v')$ are identically zero.    On the
other hands,   $(\bcr 2 \ecr)$ is equivalent to their vanishing  after  the specialization $L\mapsto \ell$, $L\mapsto \ell'$ for every general $\ell, \ell' \in S_1$.   Again this is equivalet to the vanishing of  all the $M(y_i, y'_i, \underline z,  \underline v,\underline z',  \underline v')$  (Lemma \ref{ZariskiDense}).

\medskip

It remains to prove that {\it(3)} implies {\it(2)}.   Note that the existence of a triple  $(\ell_0 \underline z_0, \underline v_0)$ such that $\wedge (\ell_0 \underline z_0, \underline v_0)$ is proved in  Lemma \ref{prop:nonvanishing1}.   Let $a$ be the invertible element of $A$  that we obtain computing
 $\wedge (\ell_0 \underline z_0,  \underline v_0)$ in $\wedge^{\p(R+1)} S_{A,R+1}/I_{A,R+1}$. Let us consider  any element 
\begin{equation} \label{eq:lesswedge}  \wedge (\ell_1 \underline z_1,  \underline v_1)\   \cdot\  
 \wedge  (\ell_2\underline z_2,  \underline v_2)   - \wedge (\ell_2 \underline z_1,  \underline v_2)  \  
\cdot \  \wedge ( \ell_1\underline z_2,  \underline v_1). \end{equation}
From the vanishing of the element in  {\it(3)} in which we set    $ (\ell,\underline z,\underline v) =(\ell_1,\underline z_1,\underline v_1) $ it follows
 $  {\wedge ( \ell_1 \underline z_1,  \underline v_1) }   =   a^{-1} \cdot  {\wedge ( \ell_1 \ \underline z_0,  \underline v_1})\cdot
  { \wedge (  \ell_0\underline z_1,  \underline v_0) }.$
We get  three similar relations  by setting  $(\ell,\underline z,\underline v)$   equal  respectively to $(\ell_2,\underline z_2,\underline v_2)$, 
   $(\ell_2,\underline z_1,\underline v_2)$ and $(\ell_1,\underline
   z_2,\underline v_1)$.  Substituting these four relations in
   \eqref{eq:lesswedge} we get $0$.


\end{proof}

\begin{theorem}\label{th:intersection}
Let  $A$ be a local $\kk$-algebra and  $\alpha \colon \spec{A}\fd \mathbf E$. 
Then 

$\alpha$   factorizes through ${\mathbf E}\cap {\mathbf \Delta}$ if and only if,
  for a general $\ell \in S_1$,  the quotient  $J_\ell:=S_{A,R}/(I_{A,R+1}\colon \ell )$ is  free  of rank $ p(R)$ and    does not depend on $\ell$.
\end{theorem}
\begin{proof}
By hypothesis and Theorem \ref{prop:generalform1},  $S_{A,R+1}/I_{A,R+1}$ is  free of rank $ p(R+1)$ and $ J_\ell$  is free of rank  $p(R)$ for a general $\ell \in S_1$.
  Therefore, it remains to prove that $\alpha$ factorizes also through $\mathbf \Delta$ if  and only if,  for a general $\ell \in S_1$, $ J_\ell$ does not depend on $\ell$ and for this we will  exploit  Lemma~\ref{lem:equiv}.

\medskip

For a general $\ell \in S_1$, we can
identify $J_\ell$ through its \Pl\ coordinates in
$\GrassScheme{S_{R}}{\p(R)}$: we choose an isomorphism $f_\ell\colon
\wedge^{\p(R)}J_\ell \xrightarrow{\sim} A $ and set, for every tuple
$\underline z$ of $p(R)$ monomials in $S_R$,   $P_{\underline
  z}(J_{\ell})=f_\ell(\wedge \underline z)$.

 Let us choose a tuples  $ \underline z_0, \underline v_0$ such that  $\wedge (\ell \underline z_0,   \underline v_0) $ is invertible for  general $\ell \in S_1$ and let $\ell_0$ be one of them  (Lemma \ref{prop:nonvanishing1}).  Then,  also   $\underline z_0$ is a basis for $J_\ell$ and $P_{\underline z_0}(J_{\ell})$ is invertible.  For every tuple    $ \underline z$ there is a suitable   matrix $\mathcal D_{\ell,\underline z} $ with entries in $A$  such that   $\underline z=
\mathcal D_{\ell,\underline z} \cdot \underline z_0$   in   $J_\ell$.
 As a consequence 
 \begin{equation} \label{eq00} 
\wedge\underline z= \det(\mathcal D_{\ell,  z}) \wedge  \underline z_0   \hbox{  in  }  \wedge^{\p(R)}J_\ell  \quad \hbox{ and } \quad   P_{\underline z}(J_{\ell})= \det(\mathcal D_{\ell,  z})\cdot P_{\underline z_0}(J_{\ell})  \hbox{  in  }  A\end{equation} 
and,  for every tuple $\underline v$ of $\p(R+1)-p(R)$ monomials in $S_{R+1}$ (especially, for $\underline v=\underline v_0$).
 \begin{equation} \label{eq0} 
\wedge (\ell \underline z,   \underline v)= \det(\mathcal D_{\ell,
  z}) \wedge (\ell \underline z_0,   \underline v)  \quad
\hbox{in}\quad  \wedge^{\p(R+1)}S_{A,R+1}/I_{A,R+1}\simeq A
.\end{equation}  
 By substitution in  \eqref{eq0} we obtain
  \begin{equation}\label{ratio}   \wedge (\ell \underline z,   \underline v)=  P_{\underline z}(J_{\ell})\cdot  P_{\underline z_0}(J_{\ell})^{-1} \cdot \wedge (\ell \underline z_0,   \underline v) \quad  \forall \underline z, \underline v, \hbox { and $\ell$ general}.  \end{equation} 
  
We use this equality to replace    $(\ell \underline z,  \underline v)$ and  $(\ell_0 \underline z,  \underline v_0)$ in  $ \wedge (\ell \underline z,  \underline v)\   \cdot\  { \wedge (  \ell_0\underline z_0,   \underline v_0) }  -   {\wedge ( \ell_0 \underline z,  \underline v_0)}\ \cdot
\ { \wedge (  \ell\underline z_0,  \underline v) }$ and find that the  condition {\it (3)} of Lemma \ref{lem:equiv} is equivalent to the vanishing of

\begin{equation}\label{eq:quasi} \left( P_{\underline z}(J_{\ell})\cdot  P_{\underline z_0}(J_{\ell})^{-1}- P_{\underline z}(J_{\ell_0})\cdot  P_{\underline z_0}(J_{\ell_0})^{-1}\right) \cdot \wedge (  \ell\underline z_0,   \underline v)  \cdot \wedge (  \ell_0\underline z_0,   \underline v_0) . \end{equation}

If we assume that  $J_\ell$ does not depend on $\ell$ for a general
$\ell$, then in particular $J_\ell=J_{\ell_0}$ (recall that $\ell_0$
is general too) and   \eqref{eq:quasi} vanishes. Hence,  $\alpha$
factorizes   through $\mathbf \Delta$ since  the condition of Lemma
\ref{lem:equiv}{\it (3)}  is fulfilled.

\medskip
On the other hand,  if  we assume that $\alpha$ factorizes through
$\mathbf \Delta$, then  \eqref{eq:quasi} vanishes   for every tuple
$\underline v$ and for general $\ell\in S_1$. We consider the special
case of  \eqref{eq:quasi} with   $\underline v=\underline v_0$ and
denote by $U$ a non-empty open subset of $S_1$ such that for all $\ell
\in U $, the following two conditions are satisfied: 
\begin{gather*}
 \left(
 P_{\underline z}(J_{\ell})
\cdot  P_{\underline z_0}(J_{\ell})^{-1} -  P_{\underline z}(J_{\ell_0})\cdot 
   P_{\underline z_0}(J_{\ell_0})^{-1}\right)   \cdot \wedge
 (  \ell\underline z_0,   \underline v_0)  \cdot \wedge (  \ell_0\underline z_0,   \underline v_0) =0
 \\
 \wedge(  \ell\underline z_0,   \underline v_0)   \hbox{   is invertible } .
\end{gather*} 

 Therefore, $   P_{\underline z}(J_{\ell})\cdot  P_{\underline
   z_0}(J_{\ell})^{-1}- P_{\underline z}(J_{\ell_0})\cdot
 P_{\underline z_0}(J_{\ell_0})^{-1}  =0$, for every tuple $\underline
 z$  and for every $\ell\in U$.  Thus   $J_\ell$ and $J_{\ell_0}$
 coincide since they have the same \Pl\ coordinates up to
 an invertible element.  We conclude that for every $\ell
 $ in $U$, $J_\ell$ does not depend on $\ell$. 
\end{proof}

\begin{remark}[Cross product remark]\label{crossProductRemark}
  In the previous theorem, we proved that the quadratic equations of
  $\Delta$ characterize that $J_l$ does not depend on $l$. In this
  remark, we explain heuristically why it is clear that the
  independency of $J_l$ with respect  to $l$ can be
  characterized by quadratic equations.

Consider the generic
  linear form with indeterminate coefficients ie. $L=a_0x_0+\dots
  a_nx_n$. Then the Plucker coordinates of $J_L$ are computed with
  the indeterminates, ie. they are elements in
  $k(L)=k(a_0,\dots,a_n)$. If $J_l=J$ is independent of $l$, then
  $J_L=J_l=J$ and the coordinates of $J_L$ turn out to be in $k$
  rather than in $k(L)$. Now,  
  a $k(L)$-point of $\PP^n$ is a $k$-point when quadratic 
  cross product equations hold.
  For instance, the point
  $P=(3a_1+2a_0a_2:6a_1+4a_0a_2:9a_1+6a_0a_2)=(1:2:3)\in \PP^2$ is  a
  $k$-point.  The coefficients $(3:6:9)$
  and $(2:4:6)$ of the graded parts are proportional. This is measured by
  determinants of order $2$. Similarly, the equations of $\Delta$ are
  the equations of degree $2$
  corresponding to the proportionality of the graded parts of $J_L$,
  hence they
  characterize that $J_L$ has coefficients in $k$. 

  In fact, this is using the ideas of the present remark that we found
  the equations of $\Delta$, after an explicit computation of
  $J_L$ in a previous version of this paper (arXiv:1612.03074v3).
  However, this approach involved many
  technical details associated with generic points. In the present version,
  we have simplified and suppressed generic points. The price
  paid for this simplification is that we only ``check''  the equations of $\Delta$
  by a local computation. It is not clear how one could have guess the
  equations without
  the use of generic points. 
\end{remark}

\section{The proof of Theorem B}
\label{sec:equations-degree-twoRED}

In theorem $B$, the linear equations are the equations of $E$ and the
quadratic equations are the equations of $\Delta$. Thus the following
claim will conclude
the proof of theorem $B$.
\\
 {\bf Claim} \label{prop:generalform}  The Hilbert scheme  $\hilbScheme{\PP^\n}{\p}$ is the subscheme $\mathbf E\cap \mathbf \Delta$ of $\GrassScheme{S_{R+1}}{p(R+1)}$.
 \medskip

We use the following easy lemma.  
\begin{lemma}\label{lemma:subfunctors}
Let $P$ and $Q$ be  closed  subschemes of a noetherian $\kk$-scheme \noindent $G$. Then, $P$ is a subscheme of $ Q$ if and only if  for every  noetherian 
local $\kk$-algebra  $A$ every map $\alpha\colon \spec{A} \fd G$  that factorizes through $Q$  also factorizes through $P$.
 \end{lemma}

\medskip

\begin{proof}Recall that as proved in 
Theorem \ref{prop:generalform1},    a  map $\alpha_{R+1} \colon \spec{A} \fd  \GrassScheme{S_{R+1}}{p(R+1)}$ ($A$ local)  factorizes through 
$\mathbf E$ if and only if for a general linear form $\ell\in S_1$ the colon ideal $I_{A,R,\ell}:=(I_{A,R+1}\colon \ell)$ gives a map $\alpha_{R,\ell}\colon 
\spec{A} \fd  \GrassScheme{S_{R}}{p(R)}$. Furthermore, $\alpha_{R+1}$   also factorizes  through $\mathbf \Delta$  if and only if for
$\ell$ general  $I_{A,R,\ell}$ is independent of $\ell$.

  As the equality of subschemes of a given scheme is a local property (Lemma \ref{lemma:subfunctors}), we consider  a noetherian local $k$-algebra $A$ and a map 
 $\alpha\colon \spec{A} \fd \GrassScheme{S_{R+1}}{\p(R+1)}$ and prove that $\alpha$ 
    factorizes through ${\hilbScheme{\PP^\n}{\p}}$ if and only if 
$\alpha$   factorizes through ${\mathbf E}$ and     $(I_{A,R+1}\colon \ell )= (I_{A,R+1}\colon S_{1} )$ for a general $\ell \in S_1$ (Theorem \ref{th:intersection}).
\medskip

First, we assume that $\alpha$ factorizes through $\hilbScheme{\PP^\n}{\p}$.

  Following the description of the Hilbert scheme given in Theorem \ref{constructionHilbertDegreeRRPlusOne}, to    $\alpha\colon \spec{A} \fd \hilbScheme{\PP^\n}{\p} \fd \GrassScheme{S_{R+1}}{\p(R+1)}$  corresponds  ${\alpha_R}\colon \spec{A} \fd \hilbScheme{\PP^\n}{\p} \fd \GrassScheme{S_{R}}{\p(R)}$ given by  
 a submodule     $I_{A,R}$ of  $S_{A,R}$ such that   $S_{A,R}/I_{A,R}$ is (locally) free of rank $\p(R)$ and $S_1 I_{A,R} \subset I_{A,R+1}$.
  
 Then,  $(I_{A,R+1}\colon \ell) \supset (I_{A,R+1}\colon S_1)\supset I_{A,R}$, so that   there is a surjective map  $  S_{A,R}/ I_{A,R}\fd S_{A,R}/(  I_{A,R+1}\colon \ell)  $. 
As the computation of  exterior powers  preserves the surjectivity
\cite[Proposition A2.2 d]{eisenbud95:commutativeAlgebraWithAView},  we
obtain the exact sequence   $0=  \wedge^{\p(R)+1} S_{A,R}/ I_{A,R}\fd
\wedge^{p(R)+1}S_{A,R}/(  I_{A,R+1}\colon \ell)  \fd 0$ and then the
vanishing of $\wedge^{p(R)+1}S_{A,R}/(  I_{A,R+1}\colon \ell)$.  By
Definition \ref{def:functorH} and  Notation \ref{rk:uguale} the  map $\alpha$ also  factorizes through ${\mathbf E}$.

\smallskip

  Now we prove the equalities  $(I_{A,R+1} \colon \ell)=(I_{A,R+1}
  \colon S_{1})=I_{A,R}$ for every general $\ell$ in $S_1$.
If $I_A$ is  the saturation  of  the ideal generated by $I_{A,R+1}$, then $I_{A,R}$ and $I_{A,R+1}$  are its homogeneous components of degree
$R$ and $R+1$ respectively. 
By the assumption on the noetherianity of
$A$,  the ideal  $I_A$ has a primary reduced decomposition   $I_A=\bigcap \mathfrak q_i$.
Moreover,  no   associated prime $\mathfrak p_i=\sqrt{\mathfrak q_i}$
contains all the variables $x_0, \dots, x_n$, since   $I_A$ is
homogeneous and saturated.
Hence,  
 the open subset of linear forms $U= \cap_i (S_1\setminus \mathfrak p_i)$ is non-empty   and  for every   
$\ell \in U$ we have  $(I_A\colon \ell)=(\bigcap \mathfrak q_i \colon \ell)=\bigcap (\mathfrak q_i \colon \ell)= \bigcap \mathfrak q_i=I_A$.
 
 Note that  ${\alpha}_R$ is indeed the map $\alpha_{R,\ell}$ for every general $\ell \in S_1$.

Therefore,  for a general $\ell\in S_1$,
$J_\ell:=S_{A,R}/(I_{A,R+1}\colon \ell )=S_{A,R}/I_{A,R}$   does not
depend on $\ell$. 

\smallskip

To prove the converse we assume that $\alpha$ factorizes through $\mathbf E\cap \mathbf \Delta $ and prove that it also factorizes through $\hilbScheme{\PP^\n}{\p}$  exploiting  the description given in  Theorem   \ref{constructionHilbertDegreeRRPlusOne}.

 By Theorems \ref{prop:generalform1}  and  \ref{th:intersection}, for a general $\ell \in S_1$ the module  $J_\ell:=S_{A,R}/(I_{A,R+1}\colon \ell )$ is (locally)  free  of rank $p(R)$ and  does not depend on $\ell$: let us denote it by $J$  and by $\tilde I_{A,R}$ the kernel of  the canonical map $S_{A,R}\rightarrow J$. Then, $\ell \tilde I_{A,R}\subseteq  I_{A,R+1}$ for every $\ell$ in a suitable open subset $U$ of $S_1$. A the ground field $k$ is infinite, $U$ is not contained in a proper $k$-subvector  space of $S_1$, hence $U$ contains a  basis $\ell_0, \ell_1, \dots, \ell_n$ of $S_1$. By construction $\ell_i \tilde I_{A,R}\subseteq  I_{A,R+1}$, so that $S_1 \tilde I_{A,R}\subseteq  I_{A,R+1}$.

 As   $S_{A,R}/\tilde I_{A,R}=J$  is free of rank $\p(R)$ and $S_1 \tilde I_{A,R}\subseteq  I_{A,R+1}$, by Theorem  \ref{constructionHilbertDegreeRRPlusOne} the  pair $(\tilde I_{A,R},  I_{A,R+1})$ corresponds to $\alpha \colon \spec A \rightarrow  \hilbScheme{\PP^\n}{\p}$.

\medskip

Finally, we conclude the proof of Theorem B, (proved so far for $k$ algebraically closed), holds for any base field $k$.  

Let $K$ be the algebraic closure  of $k$  and consider
  the inclusion $P_k:=H^0_*\mathcal O_{\PP_k^{N-1}}\subset  P_K:= H^0_*\mathcal O_{\PP_K^{N-1}}$. According to 
\cite[Proposition 1.3.10]{ega1Springer}, $\hilbScheme{\PP^n_K}{p}=
\hilbScheme{\PP^n_k}{p} \times_{\spec{k}} {\spec{K}}$,   hence any set of equations in $P_k$ which define 
the Hilbert scheme $\hilbScheme{\PP^n_K}{p}\subset
\PP_K^{N-1}$ over $K$
are equations defining the Hilbert scheme $\hilbScheme{\PP^n_k}{p}
\subset \PP^{N-1}_{k}$ over $k$. Since the equations of Theorem B are
defined over $\mathbb Z$ or  $\mathbb Z/n \mathbb Z$ ($n$ being the
characteristic of $k$), hence they belong to  $P_k$, and are valid
on the algebraic closure $K$, the equations also define the Hilbert scheme  over $k$. 
\end{proof}

\section{Extensions of the theorem and an example}
\label{sec:changing-base-field}

The goal of this section is to discuss  the hypothesis   $R\geq r $  
that we assume in  Theorem B. We prove that it is not possible to
use the same set of equations for $R=r-1$ and we explain this phenomenon
in terms of Castelnuovo-Mumford regularity. 

We first show that in general the minimal  standard embedding   $j_r: \ \hilbScheme{\PP^n_k}{p}\hookrightarrow \GrassScheme{S_{r}}{p(r)}$ cannot be  defined by  quadratic equations. Therefore, a fortiori, the  bound we assume on $R$  is sharp. 

We  present in details an explicit example.

\begin{example} \label{nonbasta2}
Each closed point of the Hilbert scheme   $\hilbScheme{\PP^2}{t+2} $   is  either the disjoint union of a line and a point like the one given by the ideal $(xy,xz)$, 
or a line with an embedded   point like the one given by the ideal $(x^2,xy)$.

Note that for any given line  $\ell$ and  point $P\in \ell$,  there is
only one saturated ideal with Hilbert polynomial $t+\bco 2 \ecr$ whose associated primes are those corresponding to $\ell$ and $P$. For instance $(x^2,xy)$ is the only one with associated primes  $(x)$ and $(x,y)$.

  Therefore, by easy arguments, we can see that  $\hilbScheme{\PP^2}{t+2} $  is 
 isomorphic to $\PP^2 \times {\PP^2}^\vee$. Hence, it can be embedded in a projective space as a subscheme cut by quadrics, as  for instance in $\PP^8$ 
by the Segre embedding. However, here we  are interested in the  standard embedding  $\hilbScheme{\PP^2}{t+2} \hookrightarrow \GrassScheme{S_{2}}{4} 
\hookrightarrow \PP^{14}$  of Theorem \ref{constructionHilbertDegreeRRPlusOne}.

  Using  the   computational methods developed in  \cite{BCR}   and   \cite{BLMR}  we obtain  that $\hilbScheme{\PP^2}{t+2} $  is the subscheme of $\PP^{14}$  
defined by an  ideal $\mathcal I$  in  $k[\Delta]$ (where $\Delta $ denotes  the $15$  \Pl\ variables) having as a set of minimal generators 
$15$ \Pl\ relations,       $15$  additional quadrics  and    $28$ cubics.  In the appendix we list  this   set of  equations (except the \Pl\ ones).

By standard computational methods, we checked  that $\mathcal I$ is saturated, hence it contains all the quadrics of $\PP^{14}$  
 that vanish on  $\hilbScheme{\PP^2}{t+2} $,  and computed the  Hilbert polynomial of $k[\Delta]/\mathcal I$ (namely the Hilbert polynomial of $\hilbScheme{\PP^2}{t+2} $ in $\PP^{14}$):
$$P_{\mathcal I}(Z)={{\it Z}}^{4}+\frac{9}{2}\,{{\it Z}}^{3}+7\,{{\it Z}}^{2}+\frac{9}{2}\,{\it Z}
+1.$$
We also computed the Hilbert polynomial of 
$k[\Delta]/\mathcal Q$,  where $\mathcal Q:=(\mathcal I_2)$   is the ideal generated by all the 30 quadrics  vanishing on $\hilbScheme{\PP^2}{t+2} $,  founding a different polynomial
$$P_{\mathcal Q}(Z)={{\it Z}}^{4}+\frac{15}{2}\,{{\it Z}}^{3}+\frac{3}{2}\,{{\it Z}}^{2}+3\,{\it Z}+2. $$

This shows not only that $\mathcal Q \neq \mathcal I$, but also that $\mathcal Q^{sat} \neq \mathcal I$, namely that $\hilbScheme{\PP^2}{t+2} $ 
 embedded  in $\PP^{14}$ is not cut out by quadrics, while by  Theorem B this is true if we consider any standard embedding 
  $\hilbScheme{\PP^2}{t+2} \hookrightarrow \GrassScheme{S_{R+1}}{p(R+1)}$ with $R+1\geq r+1=3$.

\medskip

We highlight two interesting aspects  of   this example and the related computations presented in the appendix.

The first one is about  the minimal embedding   to which the equations obtained in this paper apply, namely    $\hilbScheme{\PP^2}{t+2} \hookrightarrow \GrassScheme{S_{3}}{5}\hookrightarrow \PP^{251}$. 
After Theorem B we know that $\hilbScheme{\PP^2}{t+2}$ is  a degenerate subscheme of $\PP^{251}$  contained in the hyperplanes defined by  the linear forms  $E(\underline x,{\underline z},{\underline v} )$. 
A priori there are  $126=21\cdot 6$ of them, since  we can choose the monomial $\underline x\in k[x,y,z]_5$ in       $21$ ways and  the  set $\underline z$ of  5 distinct monomials in $k[x,y,z]_2$ in $6$ ways (while $\underline v$ is empty). In the paper we do not prove  that  the  linear  forms   $E(\underline x,{\underline z},{\underline v} )$   are always   independent. However, in the present case  they are, as we checked   by  explicit computations. 

    In the appendix we list  a basis for the vector space generated by  the $126$  linear forms. 

\medskip

The second remark concerns the embedding  $\hilbScheme{\PP^2}{t+2}
\hookrightarrow \GrassScheme{S_{2}}{4}\hookrightarrow \PP^{14}$.  In
Example \ref{nonbasta2}  we  proved that the quadratic equations are
not sufficient, but we also show  that  15  quadratic equations
(independent from the \Pl\ ones) do exist. Nevertheless, the sets of
equations  given  by Theorem B {\it  2),  3)} in this case are
empty. Indeed,  it is easy to check even from the simplified
description given in the introduction,  that for every subvector space
$W $ in $S_2$ of codimension $2+2=4$ a  general linear forms $\ell$
satisfies the equality $(W:\ell)=(W:S_1)=\{0\}$, and
the codimension  of $(W:\ell)$ in $S_1$ is $3$;   hence all the points of the Grassmannian satisfy  the conditions described by  the equations of Theorem B.  
Therefore, the equations given by {\it  2),  3)} are not sufficient to describe  $\hilbScheme{\PP^2}{t+2} $ in $ \GrassScheme{S_{2}}{4}$, since they do not exclude for instance the point of the Grassmannian corresponding to $W=(x^2,z^2)_2$, though the Hilbert polynomial of $\Proj(k[x,y,z]/(x^2, z^2)) $  is $4$ and not $t+2$. The following proposition generalizes this observation.

\end{example}

\begin{proposition}\label{semifinale}
Let $p(t)$ be any Hilbert polynomial of subschemes of $\PP^n$, with the sole exclusion of the Hilbert polynomial of 
 $\Proj \left( k[x_0, \dots, x_n]/(x_0, \dots, x_{s-1}, x_s^r)\right)$, for every $s, r$.

Then, the equations of Theorem B in the case $R=r-1$ define a subscheme  $Y$  of  $ \GrassScheme{S_{r}}{p(r)}$ that properly  contains $\hilbScheme{\PP^n}{p} $ 
as a set. 
\end{proposition}

\begin{proof}
First of all we observe that also for $R=r-1$ the equations are
satisfied by all the points of the Hilbert scheme. Indeed, for every
saturated ideal $\mathfrak b$, every integer $m$ and  general linear
form $\ell$, then  $(\mathfrak b_m\colon S_1)=(\mathfrak b_m\colon
\ell)$ as shown in the proof of theorem B. 

Moreover, if $\mathfrak b$ is a saturated ideal with Hilbert
polynomial $p$, then its regularity is $\leq r$ and its Hilbert
function coincides with the Hilbert polynomial also in degree
$r-1$ (see for instance \cite[Remark 2.4]{CMR}). 
 We then  prove the statement showing that there is an ideal $I$ in
 $S$  whose Hilbert polynomial is different from $p$, while   $I_r$
 satisfies the conditions corresponding to the equations of Theorem B
 in the case $R=r-1$.   

Let $\prec$  be the term order Lex in  $k[x_0, \dots, x_n]$ with $x_0 \succ\dots \succ x_n$ 
and  let $J$ be the saturated (non-irrelevant)  Lex-segment ideal with Hilbert
polynomial $p$. It is well known
that
 $J$ has regularity exactly $r$ and that no monomial in its minimal  monomial basis $B$ is divisible by $x_n$; moreover  the $\prec$-minimal monomial $x^\alpha$ 
in  $B$ has degree $r$ and  minimal variable $x_m$  larger than $x_n$ ($x^\alpha$ is a minimal generator of a saturated monomial ideal) and  strictly lower than the maximal variable    $x_M$ of $x^\alpha$ (the cases with  $x_m=x_M$ are those excluded in the statment).

Let $I'$ be the  ideal generated by $B^\ast:=B \setminus \{x^{\alpha} \}$ and let $I$ be the  ideal generated  by $B':=B^\ast \cup \{x_m^r\}$.  By construction $I'$ is a saturated lex-segment ideal and $I, J\supset I'$ with equalities $I_s= J_s= I_s$ for every $s\leq r-1$.

\medskip

Now we prove that  $I$ is saturated. Let $x^\gamma $ be a monomial  in  the saturation of $I$.  Then,  for $v$ sufficiently large,  $x^\gamma x_n^v $  is divisible by some monomial in
 $B'$.  If  $x^\gamma x_n^v $   is divisible by  a monomial in $B^\ast$, then it is an element of the saturated ideal $I'$, so that   $x^\gamma\in I'\subset I$. 
On the other hand, if $x^\gamma x_n^v $   is divisible by $x_m^r$, then also $x^\gamma $ is, so  that $x^\gamma \in I$. Therefore $I$ coincides with its saturation.

\medskip

By construction, $I_r$ is the vector space generated by $\dim(S_r)-p(r)$  monomials of degree $r$ (those of $J_r$ with the only exception that  $x_m^r$ replaces $ x^{\alpha}$),  
so that  
its codimension is $p(r)$.Moreover,   
 for a general linear form $\ell$  since  $I$ is saturated,   and
 $(I_r\colon S_1)=(I'_r\colon S_1)=I'_{r-1}=J_{r-1}$ since the only
 monomial in $I_r\setminus I'_r$ is a minimal generator of $I$, hence
 it cannot be divisible  by  a monomial in $I_{r-1}=I'_{r-1}$.
Since  $(I_r\colon S_1)\subset (I_r\colon \ell)$, whereas $\dim
(I_r\colon S_1)\geq (I_r\colon \ell)$ by Green's theorem and the above
computation, we have  $(I_r\colon S_1)= (I_r\colon \ell)$, ie.
 $I_r$ admits  general linear forms in the sense of Theorem \ref{th:intersection}.
\medskip

We  conclude showing that  the scheme $X$ defined by  $I$ is  not   a $k$-point of
$\hilbScheme{\PP^n}{p(t)} $.   If the Hilbert polynomial of $X$ were
$p(t)$,  then the   dimension of  $I_{r+1}$  should be equal to that
of  $J_{r+1}$.
We now prove that, on the contrary, $\dim(I_{r+1})>\dim(J_{r+1})$.  As $I$ and $J$ are  monomial ideals, we
compare  the two dimensions by comparing  the number of monomials
in $I_{r+1}$ and $J_{r+1}$.  Recall that the  monomial bases of $J$ and $I$ are 
$B = B^\ast \cup\{x^{\alpha}\}$  and $B '= B^\ast \cup\{x_m^r \}$
 and that  $J$  is the saturated  lexsegment ideal with regularity $r$, $x^\alpha$ 
has degree $r$ and it is the $\prec$-minimal monomial in  $B$, 
its  maximal variable being 
 $\prec$-larger than  the minimal variable $x_m$.

Both $J_{r+1}$ and $I_{r+1}$ contains the monomials $S_1B^\ast$. Then it is sufficient to consider those that are not in this set.  As $J$ is a lexsegment ideal, we know that there are exactly $n-m+1$ monomials in $J_{r+1}\setminus S_1B^\ast$ (they are $ x_jx^\alpha$ with $j=m, \dots, n$). 

We  now  observe that these monomials are $\prec$-lower that all those in 
$S_1B^\ast$ and that  by construction  $x_mx^\alpha \succeq
x_Mx_m^r$.  Therefore, the $n-M+1$ monomials $x_jx_m^r\in  I_{r+1}$
with $j=x_M, \dots, x_n $, do not belong to $S_1B^\ast$. Then $\dim
( I   _{r+1})-\dim (J_{r+1}) \geq m-M>0$.
\ecr
\end{proof}
\ecr

In  the proof of Proposition \ref{semifinale} we present for almost
every Hilbert polynomial $p$ an explicit example of  an ideal $I$ is $S$ such that $I_r$ is  a $k$-point of $
\GrassScheme{S_{r}}{p(r)}\setminus \hilbScheme{\PP^n}{p} $ that
satisfies the equations of Theorem B in the case $R=r-1$.   We observe
that the regularity  of the scheme defined by $I$  is  equal to  or larger   than
the Gotzmann number $r$ of $p$.   It is not by chance that this
happens; indeed the same happens for every    ideal $I$  such that    $I_r$ is   a $k$-point of $
\GrassScheme{S_{r}}{p(r)}\setminus \hilbScheme{\PP^n}{p} $ that    satisfies the  equations given by
Theorem B for $R=r-1$. In fact   
  the codimension of  $I_t$ in $S_t$   coincides with  $p(t)$ for
both $t=r-1$ and $t=r$; if we also assume that the regularity of $I$ is
at most $r-1$, the above conditions imply that the Hilbert polynomial of
$k[x_0, \dots, x_n]/I$ is $p$.  Therefore, in a suitable open subset
of the Grassmannian $ \GrassScheme{S_{r}}{p(r)}$ the equations given
in Theorem B in the case $R=r-1$  define   the open subscheme of
Hilbert scheme  
  $\hilbScheme{\PP^n}{p} $ where the regularity is upper bounded by $r-1$.

Following \cite{BBR}, we denote by    $\hilbScheme{\PP^n}{p,[r']} $ the  {\it Hilbert scheme with regularity upper bounded by $r'$},  namely the open subscheme of $\hilbScheme{\PP^n}{p} $  that parametrises the flat families with  regularity lower than or equal to $r'$; for the main features of this scheme we refer to      \cite{BBR}  and the references therein.

\begin{theorem}\label{reglim} Let $p$ be any Hilbert polynomial of subschemes in $\PP^n$ and $r$  be its Gotzmann number. Let moreover $r'$  and   $R$ be  integers such that   $r'\leq R\leq  r-1$.

Then, the Hilbert scheme $\hilbScheme{\PP^n}{p,[r'] } $
is the locally closed subscheme of   $ \PP^{D(R+1)}$ defined by the equations given in  Theorem B  and a suitable  set of linear inequalities.
\end{theorem}

\begin{proof} The result is  a  straightforward   consequence of the
  results of \cite{BBR} and of  \S  7 about the meaning of the linear
  and quadratic equations of Theorem B.   We outline the  proof
  simply considering $k$-points, but the arguments can be
  generalized to families.


By   \cite[Theorem 1.2 (ii)]{BBR},  for every  integer $m  \geq r'$,    $\hilbScheme{\PP^n}{p,[r'] } $
 can be embedded as a closed subscheme of   $
  \GrassScheme{S_{m}}{p(m)} \setminus L^{r',m}_{p}$ where    $L^{r',m}_{p}$   is a subscheme  of 
$ \GrassScheme{S_{m}}{p(m)}  $ cut by 
a linear space under the \Pl\  embedding. 

Moreover,  by  \cite[Lemma 7.1]{BBR},  the  $k$-points of
$\hilbScheme{\PP^n}{p,[r'] } $ are exactly the $k$-points    $V$  of
$\GrassScheme{S_{m}}{p(m)} \setminus L^{r',m}_{p} $ such  that the ideal
$I:=(V)^{sat}\subset k[x_0, \dots, x_n]$   satisfy  the condition
${\rm  codim}(I_t)=p(t)$ for  every integer $t\geq r'$  (so that in particular $I_{m}=V$);   this condition is equivalent to the apparently slighter condition ${\rm  codim}(I_t)=p(t)$  for two consecutive integers $t=t_0\geq r'$ and
$t=t_0+1$.  Here we are interested in the case $m   = R+1$ and  describe the $k$-points of $\hilbScheme{\PP^n}{p,[r'] } $ as $k$-points    $V$  of
$\GrassScheme{S_{R+1}}{p(R+1)} \setminus L^{r',R+1}_{p} $  such that $I:=(V)^{sat}$ satisfy the above condition for $t=t_0=R$ and $t=t_0+1=R+1$. 


\medskip

 In \S 7 we proved 
 that a 
$k$-point  $V$   of  $\GrassScheme{S_{R+1}}{p(R+1)} $  satisfies  the
equations of Theorem B  if and only if the ideal   $I:=(V)^{sat}$  satisfies
the conditions  ${\rm  codim}(I_{R+1})=p(R+1)$  and   ${\rm
  codim}(I_{R})={\rm
  codim}(I_{R+1}\colon S_1)={\rm
  codim}((I_{R+1}\colon \ell))= p(R)$ for a general $\ell \in S_1$.

Therefore, for every  $k$-point $V$ of $\GrassScheme{S_{R+1}}{p(R+1)}
\setminus L^{r',R+1}_{p} $,  we see that $V$ is a point of
$\hilbScheme{\PP^n}{p,[r'] } $ if and only if $I:=(V)^{sat}$ satisfies the equations of Theorem B.
\end{proof}

\bibliographystyle{plain} 

\begin{thebibliography}{30}

\bibitem{ABM}
M.~Alonso, J.~Brachat and B.~Mourrain.
\newblock The {H}ilbert scheme of points and its link with border basis.
\newblock Available at {\texttt{arxiv.org/abs/0911.3503v2}}

\bibitem{atiyahMacdonal:introductionToCommutativeAlgebra}
M.~F. Atiyah and I.~G. Macdonald.
\newblock {\em Introduction to commutative algebra}.
\newblock Addison-Wesley Publishing Co., Reading, Mass.-London-Don Mills, Ont.,
  1969.

\bibitem{bartocciBruzzoLanzaRava}
C.~Bartocci, U.~Bruzzo, V.~Lanza, and C.L.S.~Rava.
\newblock {H}ilbert schemes of points of $\mathcal{O}_{\mathbb{P}^1}(-n)$ as
  quiver varieties.
\newblock Available at {\texttt{arxiv.org/abs/1504.02987}}.


\bibitem{Bayer82}
D.~Bayer.
\newblock {\em The division algorithm and the {H}ilbert schemes}.
\newblock PhD thesis, Harvard University, 1982.


\bibitem{BBR} C.~Bertone, E.~Ballico, and M.~Roggero.
 \newblock  {The Locus of Points of the Hilbert Scheme with Bounded Regularity}.
\newblock {\em Communications in Algebra}  43, (7)  2015, {2912-2931}.



\bibitem{BCR} C.~Bertone, F.~Cioffi, and M.~Roggero.
 \newblock  Macaulay-like marked bases.
\newblock {\em J. Algebra Appl.}  16,  2017 


\bibitem{BLMR}
J.~Brachat, P.~Lella, B.~Mourrain, and M.~Roggero.
\newblock Extensors and the {H}ilbert scheme.
\newblock {\em Ann. Sc. Norm. Super. Pisa Cl. Sci.}, XVI(1), 2016.

\bibitem{BrunsHerzog}
W.~Bruns and J.~Herzog.
\newblock {\em Cohen-{M}acaulay rings}, volume~39 of {\em Cambridge Studies in
  Advanced Mathematics}.
\newblock Cambridge University Press, Cambridge, 1993.

\bibitem{buloisEvain}
M.~Bulois and L.~Evain.
\newblock Nested punctual {H}ilbert schemes and commuting varieties of
  parabolic subalgebras.
\newblock {\em J. Lie Theory}, 26(2):497--533, 2016.


\bibitem{CMR}
 F.~Cioffi and  M.G.~Marinari  and L.~Ramella
\newblock Regularity bounds by minimal generators and Hilbert function.
\newblock {\em Collect. Math.  }, 60(1):89--100, 2009.



\bibitem{eisenbud95:commutativeAlgebraWithAView}
D.~Eisenbud.
\newblock {\em Commutative algebra (with a view toward algebraic geometry)}, volume 150 of {\em Graduate Texts in
  Mathematics}.
\newblock Springer-Verlag, New York, 1995.


\bibitem{EisenbudHarris}
D.~Eisenbud and J.~Harris.
\newblock {\em The geometry of schemes}, volume 197 of {\em Graduate Texts in
  Mathematics}.
\newblock Springer-Verlag, New York, 2000.

\bibitem{fultonYoungTableaux}
W.~Fulton.
\newblock {\em Young tableaux : with applications to representation theory and
  geometry}.
\newblock London Mathematical Society student texts. Cambridge University
  Press, Cambridge, New York, 1997.

\bibitem{galligo}
A.~Galligo.
\newblock Th\'eor\`eme de division et stabilit\'e en g\'eom\'etrie analytique
  locale.
\newblock {\em Ann. Inst. Fourier (Grenoble)}, 29(2):vii, 107--184, 1979.

\bibitem{gelfand-kapranov-zelevinsky94:multidemensionalDeterminants}
I.~M. Gelfand, M.~M. Kapranov, and A.~V. Zelevinsky.
\newblock {\em Discriminants, resultants and multidimensional determinants}.
\newblock Modern Birkh\"auser Classics. Birkh\"auser Boston, Inc., Boston, MA,
  2008.
\newblock Reprint of the 1994 edition.

\bibitem{gotzmann}
G.~Gotzmann.
\newblock Eine {B}edingung f\"ur die {F}lachheit und das {H}ilbertpolynom eines
  graduierten {R}inges.
\newblock {\em Math. Z.}, 158(1):61--70, 1978.

\bibitem{granger}
M. Granger.
\newblock G\'eom\'etrie des sch\'emas de {H}ilbert ponctuels.
\newblock {\em M\'em. Soc. Math. France (N.S.)}, (8):84, 1983.

\bibitem{green:macaulayGotzmann}
M.~Green.
\newblock Restrictions of linear series to hyperplanes, and some results of
  {M}acaulay and {G}otzmann.
\newblock In {\em Algebraic curves and projective geometry ({T}rento, 1988)},
  volume 1389 of {\em Lecture Notes in Math.}, pages 76--86. Springer, Berlin,
  1989.

\bibitem{grothendieck:formalismeFoncteursRepresentables}
A.~Grothendieck.
\newblock Techniques de construction en g\'eom\'etrie analytique. iv.
  formalisme g\'en\'eral des foncteurs repr\'esentables.
\newblock In {\em S{\'e}minaire Cartan, Vol.\ 13, Expos{\'e} No. 11}. 1960-1961.

\bibitem{grothendieck60:techniquesConstructionEtSchemasDeHilbert}
A.~Grothendieck.
\newblock Techniques de construction et th{\'e}or{\`e}mes d'existence en
  g{\'e}om{\'e}trie alg{\'e}brique. {IV}. {L}es sch{\'e}mas de {H}ilbert.
\newblock In {\em S{\'e}minaire Bourbaki, Vol.\ 6}, pages Exp.\ No.\ 221,
  249--276. Soc. Math. France, Paris, 1995.

\bibitem{ega1Springer}
Grothendieck, Alexander; Dieudonn\'e, Jean A. 
\newblock \'Ele\'ements de g\'eome\'etrie alge\'ebrique. I. (English) 
\newblock [B] Die Grundlehren der mathe-matischen Wissenschaften. 166. Berlin-Heidelberg-New York: Springer-
Verlag. IX, 466 p. (1971). ISBN 3-540-05113-9

\bibitem{HartshorneAG}
R~Hartshorne.
\newblock {\em Algebraic Geometry}
\newblock{Encyclopaedia of mathematical sciences},
 {Springer} {1977}.





\bibitem{haiman_sturmfels02:multigradedHilbertSchemes}
M~Haiman and B.~Sturmfels.
\newblock Multigraded {H}ilbert schemes.
\newblock {\em J. Algebraic Geom.}, 13(4):725--769, 2004.

\bibitem{iarrobino-kleiman:around_gotzmann_number}
A.~Iarrobino and V.~Kanev.
\newblock {\em Power sums, {G}orenstein algebras, and determinantal loci},
  volume 1721 of {\em Lecture Notes in Mathematics}.
\newblock Springer-Verlag, Berlin, 1999.
\newblock Appendix C by A. Iarrobino and S. L. Kleiman.

\bibitem{KleimanLaksov}
S.~L.~Kleiman and   D.~Laksov.
\newblock Schubert calculus.
\newblock {\em The American Mathematical Monthly}, 79(10):1061--1082, 1972.

\bibitem{lella-roggero:functoriality}
P.~Lella and M.~Roggero.
\newblock On the functoriality of marked families.
\newblock {\em Journal of Commutative Algebra}, 8(3):367--410, 2016.

\bibitem{nakajimaLecturesOnHilbertSchemes}
H.~Nakajima.
\newblock {\em Lectures on {H}ilbert schemes of points on surfaces}, volume~18
  of {\em University Lecture Series}.
\newblock American Mathematical Society, Providence, RI, 1999.

\bibitem{ramanathan87:equationsForSchubertVarieties}
A.~Ramanathan.
\newblock Equations defining schubert varieties and frobenius splitting of
  diagonals.
\newblock volume~65, pages 61--90.



\end{thebibliography}

\section{Appendix}

The following are $15$ quadrics \noindent $F_i$ and 28 cubics \noindent $G_j$ that, together with the \Pl\ relations, generate the saturated ideal $\mathcal I$ of $\hilbScheme{\PP^2}{t+2} \hookrightarrow  \PP^{14}$. Each  \Pl\ variable  corresponds to the choice of 4  monomials   in $k[x,y,z]_2$: we denote it by the position of the 2 missing monomials   in the list    ordered  in decreasing degrevlex  order $x^2, xy,y^2,xz, yz, z^2$.

%
%
%
%
%
%
%
%
%
%
%
%
%
%
%
%
%
%
%
%
%
%
%
%
%
%
%
%

\bigskip
\noindent $F_{{1}}=-P_{{3,5}}P_{{5,6}}+{P_{{3,6}}}^{2}$

\noindent $F_{{2}}=-P_{{2,5}}P_{{5,6}}+2 P_{{2,6}}P_{{3,6}}-
P_{{3,4}}P_{{5,6}}-P_{{3,5}}P_{{4,6}}$

\noindent $F_{{3}}=-2 P_{{2,3}}P_{{5,6}}+P_{{2,5}}P_{{3,6}}-
2 P_{{3,4}}P_{{3,6}}-P_{{3,5}}P_{{4,5}}$

\noindent $F_{{4}}=-P_{{1,5}}P_{{5,6}}+2 P_{{1,6}}P_{{3,6}}-
P_{{2,4}}P_{{5,6}}-P_{{2,5}}P_{{4,6}}+{P_{{2,
6}}}^{2}-P_{{3,4}}P_{{4,6}}$

\noindent $F_{{5}}=-4 P_{{1,3}}P_{{5,6}}+2 P_{{1,5}}P_{{3,6
}}-3 P_{{2,4}}P_{{3,6}}+P_{{2,5}}P_{{2,6}}-P
_{{2,5}}P_{{4,5}}-P_{{3,4}}P_{{4,5}}$

\noindent $F_{{6}}=-2 P_{{1,3}}P_{{3,6}}+P_{{1,5}}P_{{3,5}}+
P_{{2,3}}P_{{2,6}}-P_{{2,3}}P_{{4,5}}-P_{{2,4
}}P_{{3,5}}+{P_{{3,4}}}^{2}$

\noindent $F_{{7}}=-P_{{1,4}}P_{{5,6}}-P_{{1,5}}P_{{4,6}}+2 
P_{{1,6}}P_{{2,6}}-P_{{2,4}}P_{{4,6}}$

\noindent $F_{{8}}=-3 P_{{1,2}}P_{{5,6}}-2 P_{{1,3}}P_{{4,6
}}-2 P_{{1,4}}P_{{3,6}}+3 P_{{1,5}}P_{{2,6}}-
P_{{1,5}}P_{{4,5}}-P_{{2,4}}P_{{2,6}}-P_{{2,4
}}P_{{4,5}}$

\noindent $F_{{9}}=-4 P_{{1,2}}P_{{3,6}}+2 P_{{1,3}}P_{{2,6
}}-3 P_{{1,3}}P_{{4,5}}+P_{{1,5}}P_{{2,5}}-P
_{{2,4}}P_{{2,5}}+P_{{2,4}}P_{{3,4}}$

\noindent $F_{{10}}=2 P_{{1,2}}P_{{3,5}}-P_{{1,3}}P_{{2,5}}-
2 P_{{1,3}}P_{{3,4}}+P_{{2,3}}P_{{2,4}}$

\noindent $F_{{11}}=-P_{{1,4}}P_{{4,6}}+{P_{{1,6}}}^{2}$

\noindent $F_{{12}}=-P_{{1,2}}P_{{4,6}}-P_{{1,4}}P_{{2,6}}-
P_{{1,4}}P_{{4,5}}+2 P_{{1,5}}P_{{1,6}}$

\noindent $F_{{13}}=-P_{{1,2}}P_{{2,6}}-P_{{1,2}}P_{{4,5}}+2
 P_{{1,3}}P_{{1,6}}-P_{{1,4}}P_{{2,5}}+P_{{1
,4}}P_{{3,4}}+{P_{{1,5}}}^{2}$

\noindent $F_{{14}}=-P_{{1,2}}P_{{2,5}}+2 P_{{1,2}}P_{{3,4}}
+2 P_{{1,3}}P_{{1,5}}-P_{{1,3}}P_{{2,4}}$

\noindent $F_{{15}}=-P_{{1,2}}P_{{2,3}}+{P_{{1,3}}}^{2}$

\noindent $G_{{1}}=P_{{1,6}}{P_{{5,6}}}^{2}-P_{{2,6}}P_{{4,6}
}P_{{5,6}}+P_{{3,6}}{P_{{4,6}}}^{2}$

\noindent $G_{{2}}=2 P_{{1,6}}P_{{3,6}}P_{{5,6}}-P_{{5,6}}
P_{{2,5}}P_{{4,6}}-P_{{3,4}}P_{{4,6}}P_{{5,6}
}+P_{{3,5}}{P_{{4,6}}}^{2}$

\noindent $G_{{3}}=-4 P_{{1,3}}{P_{{5,6}}}^{2}+4 P_{{1,5}}P
_{{3,6}}P_{{5,6}}-2 P_{{5,6}}P_{{2,4}}P_{{3,6}}-
P_{{2,5}}P_{{4,5}}P_{{5,6}}-P_{{3,4}}P_{{4,5}
}P_{{5,6}}+P_{{3,5}}P_{{4,5}}P_{{4,6}}$

\noindent $G_{{4}}=-P_{{1,3}}P_{{3,6}}P_{{5,6}}+P_{{1,5}}
P_{{3,5}}P_{{5,6}}+P_{{2,3}}P_{{3,6}}P_{{4,6}
}-P_{{2,3}}P_{{4,5}}P_{{5,6}}-P_{{5,6}}P_{{2,
4}}P_{{3,5}}+{P_{{3,4}}}^{2}P_{{5,6}}$

\noindent $G_{{5}}=P_{{1,4}}{P_{{5,6}}}^{2}-P_{{1,5}}P_{{4,6}
}P_{{5,6}}+2 P_{{1,6}}P_{{3,6}}P_{{4,6}}-P_{
{2,4}}P_{{4,6}}P_{{5,6}}$

\noindent $G_{{6}}=-2 P_{{1,3}}P_{{4,6}}P_{{5,6}}+P_{{1,4}}
P_{{3,6}}P_{{5,6}}+3 P_{{1,5}}P_{{3,6}}P_{{4
,6}}-P_{{1,5}}P_{{4,5}}P_{{5,6}}-P_{{2,4}}P_{
{2,6}}P_{{5,6}}-P_{{5,6}}P_{{2,4}}P_{{4,5}}$

\noindent $G_{{7}}=P_{{5,6}}P_{{1,2}}P_{{4,6}}-P_{{1,4}}
P_{{2,6}}P_{{5,6}}+2 P_{{1,4}}P_{{3,6}}P_{{4
,6}}-P_{{1,4}}P_{{4,5}}P_{{5,6}}$

\noindent $G_{{8}}=P_{{1,2}}P_{{3,6}}P_{{5,6}}-P_{{1,3}}
P_{{2,6}}P_{{5,6}}+P_{{1,3}}P_{{3,6}}P_{{4,6}
}$

\noindent $G_{{9}}=P_{{1,2}}P_{{3,6}}P_{{4,6}}-P_{{1,3}}
P_{{1,6}}P_{{5,6}}$

\noindent $G_{{10}}=2 P_{{1,3}}P_{{3,5}}P_{{5,6}}-P_{{2,3}}
P_{{2,5}}P_{{5,6}}-P_{{2,3}}P_{{3,4}}P_{{5,6}
}+P_{{2,3}}P_{{3,5}}P_{{4,6}}$

\noindent $G_{{11}}=2 P_{{3,5}}P_{{1,2}}P_{{5,6}}-P_{{1,3}}
P_{{2,5}}P_{{5,6}}-P_{{1,3}}P_{{3,4}}P_{{5,6}
}+P_{{1,3}}P_{{3,5}}P_{{4,6}}$

\noindent $G_{{12}}=P_{{1,2}}P_{{3,4}}P_{{5,6}}+P_{{1,2}}
P_{{3,5}}P_{{4,6}}-P_{{1,3}}P_{{2,4}}P_{{5,6}
}$

\noindent $G_{{13}}=P_{{1,2}}P_{{2,6}}P_{{4,6}}-P_{{1,4}}
P_{{1,5}}P_{{5,6}}+P_{{1,4}}P_{{3,4}}P_{{4,6}
}$

\noindent $G_{{14}}=-P_{{2,5}}P_{{1,2}}P_{{4,6}}-2 P_{{1,2}}
P_{{3,4}}P_{{4,6}}+2 P_{{1,3}}P_{{1,4}}P_{{5
,6}}+P_{{1,3}}P_{{2,4}}P_{{4,6}}$

\noindent $G_{{15}}=P_{{1,4}}P_{{1,2}}P_{{5,6}}-P_{{1,2}}
P_{{1,5}}P_{{4,6}}-P_{{2,4}}P_{{1,2}}P_{{4,6}
}+2 P_{{1,3}}P_{{1,4}}P_{{4,6}}$

\noindent $G_{{16}}=-P_{{1,3}}P_{{3,5}}P_{{3,6}}+{P_{{2,3}}}^
{2}P_{{5,6}}+P_{{2,3}}P_{{3,4}}P_{{3,6}}$

\noindent $G_{{17}}=-P_{{1,2}}P_{{3,5}}P_{{3,6}}+P_{{1,3}}
P_{{2,3}}P_{{5,6}}+P_{{1,3}}P_{{3,4}}P_{{3,6}
}$

\noindent $G_{{18}}=-2 P_{{1,2}}P_{{3,4}}P_{{3,6}}-P_{{3,5}}
P_{{1,2}}P_{{4,5}}+P_{{1,3}}P_{{2,4}}P_{{3,6}
}$

\noindent $G_{{19}}=-2 P_{{1,2}}P_{{2,4}}P_{{3,6}}-P_{{2,5}}
P_{{1,2}}P_{{4,5}}-2 P_{{1,2}}P_{{3,4}}P_{{4
,5}}+4 P_{{1,3}}P_{{1,4}}P_{{3,6}}+P_{{1,3}}
P_{{2,4}}P_{{4,5}}$

\noindent $G_{{20}}=-P_{{1,3}}{P_{{3,5}}}^{2}+{P_{{2,3}}}^{2}
P_{{3,6}}+P_{{2,3}}P_{{3,4}}P_{{3,5}}$

\noindent $G_{{21}}=-P_{{1,2}}{P_{{3,5}}}^{2}+P_{{1,3}}P_{{2,
3}}P_{{3,6}}+P_{{1,3}}P_{{3,4}}P_{{3,5}}$

\noindent $G_{{22}}=2 P_{{1,2}}P_{{2,3}}P_{{3,6}}-P_{{3,5}}
P_{{1,2}}P_{{2,5}}+P_{{1,3}}P_{{2,4}}P_{{3,5}
}$

\noindent $G_{{23}}=P_{{1,2}}P_{{1,3}}P_{{3,6}}-P_{{1,2}}
P_{{1,5}}P_{{3,5}}+P_{{1,3}}P_{{1,4}}P_{{3,5}
}$

\noindent $G_{{24}}=P_{{1,2}}P_{{1,3}}P_{{4,6}}-2 P_{{1,2}}
P_{{1,4}}P_{{3,6}}+P_{{1,3}}P_{{1,4}}P_{{2,6}
}-P_{{1,3}}P_{{1,4}}P_{{4,5}}$

\noindent $G_{{25}}=-{P_{{1,2}}}^{2}P_{{3,6}}+P_{{1,2}}P_{{1,
3}}P_{{2,6}}-P_{{1,4}}P_{{1,2}}P_{{3,5}}+P_{{
1,3}}P_{{1,4}}P_{{3,4}}$

\noindent $G_{{26}}={P_{{1,2}}}^{2}P_{{4,6}}-P_{{1,2}}P_{{1,4
}}P_{{2,6}}-P_{{1,2}}P_{{1,4}}P_{{4,5}}+2 P_
{{1,3}}P_{{1,4}}P_{{1,6}}$

\noindent $G_{{27}}=2 P_{{1,2}}P_{{1,3}}P_{{1,6}}-P_{{1,4}}
P_{{1,2}}P_{{2,5}}+P_{{1,3}}P_{{1,4}}P_{{2,4}
}$

\noindent $G_{{28}}={P_{{1,2}}}^{2}P_{{1,6}}-P_{{1,2}}P_{{1,4
}}P_{{1,5}}+P_{{1,3}}{P_{{1,4}}}^{2}.$

\bigskip

The following are $126$ independent linear forms  \noindent $L_i$   that belong to the saturated ideal  of  $\hilbScheme{\PP^2}{t+2} \hookrightarrow 
 \PP^{251}$. Each  \Pl\ variable  corresponds to the choice of 5 monomials   in $k[x,y,z]_3$: we denote it by the position of these 5 monomials in the list    ordered  in decreasing degrevlex order.
 
 \bigskip

\noindent $L_{{1}}=P_{{1,2,3,5,6}}$,  \
\noindent $L_{{2}}=P_{{1,2,3,5,8}}$,  \
\noindent $L_{{3}}=P_{{1,2,3,6,8}}$,  \
\noindent $L_{{4}}=P_{{1,2,5,6,8}}$,  \
\noindent $L_{{5}}=P_{{1,3,5,6,8}}$,  

\noindent $L_{{6}}=P_{{2,3,4,6,7}}$,  \
\noindent $L_{{7}}=P_{{2,3,4,6,9}}$, \
\noindent $L_{{8}}=P_{{2,3,4,7,9}}$,  \
\noindent $L_{{9}}=P_{{2,3,5,6,8}}$,  \
\noindent $L_{{10}}=P_{{2,3,6,7,9}}$,  

\noindent $L_{{11}}=P_{{2,4,6,7,9}}$,  \
\noindent $L_{{12}}=P_{{3,4,6,7,9}}$,  \
\noindent $L_{{13}}=P_{{5,6,7,8,9}}$,  \
\noindent $L_{{14}}=P_{{5,6,7,8,10}}$,  \
\noindent $L_{{15}}=P_{{5,6,7,9,10}}$,  

\noindent $L_{{16}}=P_{{5,6,8,9,10}}$,  \
\noindent $L_{{17}}=P_{{5,7,8,9,10}}$,  \
\noindent $L_{{18}}=P_{{6,7,8,9,10}}$,  
\noindent $L_{{19}}=3 P_{{1,2,4,5,6}}+2 P_{{1,2,3,5,7}}$,  \

\noindent $L_{{20}}=-P_{{1,2,5,6,7}}+6 P_{{1,2,4,5,8}}$,  
\noindent $L_{{21}}=P_{{1,2,5,6,7}}+2 P_{{1,2,3,5,9}}$, \
\noindent $L_{{22}}=-P_{{1,2,5,7,8}}+3 P_{{1,2,3,5,10}}$,  

\noindent $L_{{23}}=2 P_{{1,2,5,7,8}}+P_{{1,2,5,6,9}}$,  \
\noindent $L_{{24}}=-2 P_{{1,2,5,8,9}}+3 P_{{1,2,5,6,10}}$,  
\noindent $L_{{25}}=2 P_{{2,3,4,5,7}}+3 P_{{1,3,4,6,7}}$, \ 

\noindent $L_{{26}}=-2 P_{{2,5,8,9,10}}+3 P_{{1,6,8,9,10}}$,
\noindent $L_{{27}}=-P_{{3,4,5,6,7}}+2 P_{{2,3,4,7,8}}$,  \
\noindent $L_{{28}}=P_{{3,4,5,6,7}}+6 P_{{1,3,4,7,9}}$,  

\noindent $L_{{29}}=-P_{{3,4,5,7,8}}+3 P_{{1,3,4,7,10}}$,  \
\noindent $L_{{30}}=-P_{{3,4,6,7,8}}+6 P_{{2,3,4,7,10}}$,  
\noindent $L_{{31}}=P_{{3,4,6,7,8}}+2 P_{{3,4,5,7,9}}$,  \

\noindent $L_{{32}}=-2 P_{{3,4,7,8,9}}+3 P_{{3,4,6,7,10}}$,  
\noindent $L_{{33}}=-2 P_{{3,5,8,9,10}}+P_{{2,6,8,9,10}}$,  \

\noindent $L_{{34}}=-P_{{3,5,8,9,10}}+3 P_{{1,7,8,9,10}}$,  
\noindent $L_{{35}}=-6 P_{{4,5,8,9,10}}+P_{{3,6,8,9,10}}$,  \

\noindent $L_{{36}}=-3 P_{{4,5,8,9,10}}+P_{{2,7,8,9,10}}$,  
\noindent $L_{{37}}=-3 P_{{4,6,8,9,10}}+2 P_{{3,7,8,9,10}}$,  \

\noindent $L_{{38}}=3 P_{{1,3,4,5,6}}+6 P_{{1,2,4,5,7}}+
P_{{1,2,3,6,7}}$,  
\noindent $L_{{39}}=-3 P_{{1,4,5,6,8}}+P_{{1,2,6,7,8}}+2 
P_{{1,2,5,7,9}}$,  \

\noindent $L_{{40}}=3 P_{{1,4,5,6,8}}-P_{{1,2,6,7,8}}+6 
P_{{1,2,4,5,10}}$,  
\noindent $L_{{41}}=3 P_{{1,4,5,6,8}}+2 P_{{1,3,5,7,8}}+
P_{{1,3,5,6,9}}$,  \

\noindent $L_{{42}}=-P_{{1,5,6,7,8}}-P_{{1,2,6,8,9}}+3 P_
{{1,2,5,7,10}}$,  
\noindent $L_{{43}}=P_{{1,5,6,7,8}}-2 P_{{1,3,5,8,9}}+3 
P_{{1,3,5,6,10}}$,  \

\noindent $L_{{44}}=P_{{1,5,6,8,9}}-3 P_{{1,2,6,8,10}}+6 
P_{{1,2,5,9,10}}$,  \
\noindent $L_{{45}}=P_{{2,3,4,5,6}}+6 P_{{1,3,4,5,7}}+3 
P_{{1,2,4,6,7}}$,  \

\noindent $L_{{46}}=P_{{2,3,4,6,8}}+2 P_{{2,3,4,5,9}}+3 
P_{{1,3,4,6,9}}$,  
\noindent $L_{{47}}=-P_{{2,4,5,6,7}}+P_{{2,3,4,6,8}}+3 P_
{{1,3,4,7,8}}$,  \

\noindent $L_{{48}}=3 P_{{2,4,5,6,8}}+2 P_{{2,3,5,7,8}}+
P_{{2,3,5,6,9}}$,  
\noindent $L_{{49}}=P_{{2,5,6,7,8}}-2 P_{{2,3,5,8,9}}+3 
P_{{2,3,5,6,10}}$,  \

\noindent $L_{{50}}=P_{{2,5,6,8,10}}-3 P_{{1,5,6,9,10}}+6 
P_{{1,2,8,9,10}}$,  
\noindent $L_{{51}}=-2 P_{{3,4,5,7,8}}-P_{{3,4,5,6,9}}+3 
P_{{1,4,6,7,9}}$,  \

\noindent $L_{{52}}=2 P_{{3,4,5,7,8}}-P_{{2,4,6,7,8}}+3 
P_{{2,3,4,6,10}}$,  
\noindent $L_{{53}}=P_{{3,5,6,8,9}}-2 P_{{2,5,7,8,9}}+3 
P_{{1,6,7,8,9}}$,  \

\noindent $L_{{54}}=P_{{3,5,6,8,10}}-P_{{2,5,6,9,10}}+3 P
_{{1,3,8,9,10}}$,
\noindent $L_{{55}}=P_{{3,5,6,8,10}}-2 P_{{2,5,7,8,10}}+3 
P_{{1,6,7,8,10}}$,  \

\noindent $L_{{56}}=-P_{{4,5,6,7,9}}-2 P_{{2,4,7,8,9}}+3 
P_{{2,4,6,7,10}}$,  
\noindent $L_{{57}}=P_{{4,5,6,7,9}}-P_{{3,4,6,8,9}}+3 P_{
{3,4,5,7,10}}$,  \

\noindent $L_{{58}}=3 P_{{4,5,6,8,9}}-2 P_{{3,5,7,8,9}}+
P_{{2,6,7,8,9}}$,  
\noindent $L_{{59}}=3 P_{{4,5,6,8,10}}-P_{{3,5,6,9,10}}+6 
P_{{1,4,8,9,10}}$,  \

\noindent $L_{{60}}=3 P_{{4,5,6,8,10}}-P_{{3,5,6,9,10}}+2 
P_{{2,3,8,9,10}}$,  
\noindent $L_{{61}}=3 P_{{4,5,6,8,10}}-2 P_{{3,5,7,8,10}}+
P_{{2,6,7,8,10}}$,  \

\noindent $L_{{62}}=P_{{4,6,7,8,9}}-6 P_{{3,4,7,8,10}}+3 
P_{{3,4,6,9,10}}$,  
\noindent $L_{{63}}=3 P_{{4,6,7,8,10}}-P_{{3,6,7,9,10}}+6 
P_{{3,4,8,9,10}}$,  \

\noindent $L_{{64}}=-P_{{1,3,5,6,7}}+3 P_{{1,3,4,5,8}}+3 
P_{{1,2,4,6,8}}+P_{{1,2,3,7,8}}$,  

\noindent $L_{{65}}=-P_{{1,3,5,6,7}}+6 P_{{1,3,4,5,8}}+3 
P_{{1,2,4,6,8}}+3 P_{{1,2,4,5,9}}$,  

\noindent $L_{{66}}=2 P_{{1,3,5,6,7}}-6 P_{{1,3,4,5,8}}-3 
P_{{1,2,4,6,8}}+P_{{1,2,3,6,9}}$,  \

\noindent $L_{{67}}=-6 P_{{1,4,5,6,8}}-2 P_{{1,3,5,7,8}}+
P_{{1,2,6,7,8}}+3 P_{{1,2,3,6,10}}$,  \

\noindent $L_{{68}}=-3 P_{{1,4,5,6,8}}-P_{{1,3,5,7,8}}+P_
{{1,2,6,7,8}}+P_{{1,2,3,8,9}}$,  \

\noindent $L_{{69}}=6 P_{{1,4,5,7,8}}+3 P_{{1,4,5,6,9}}+
P_{{1,3,6,7,8}}+2 P_{{1,3,5,7,9}}$,  \

\noindent $L_{{70}}=-P_{{1,5,6,7,8}}+P_{{1,3,5,8,9}}-P_{{1
,2,6,8,9}}+3 P_{{1,2,3,8,10}}$,  \

\noindent $L_{{71}}=2 P_{{1,5,7,8,9}}+3 P_{{1,5,6,7,10}}-3 
P_{{1,3,6,8,10}}+6 P_{{1,3,5,9,10}}$,  

\noindent $L_{{72}}=2 P_{{2,3,4,5,8}}+3 P_{{1,3,4,6,8}}+6 
P_{{1,3,4,5,9}}+3 P_{{1,2,4,6,9}}$,  

\noindent $L_{{73}}=2 P_{{2,3,5,8,10}}+3 P_{{1,5,6,7,10}}-3 
P_{{1,3,6,8,10}}+6 P_{{1,2,7,8,10}}$,  

\noindent $L_{{74}}=P_{{2,4,5,6,7}}-P_{{2,3,4,6,8}}-P_{{2,
3,4,5,9}}+3 P_{{1,2,4,7,9}}$,  

\noindent $L_{{75}}=-P_{{2,4,5,6,8}}-6 P_{{1,4,5,7,8}}-3 
P_{{1,4,5,6,9}}+P_{{1,2,6,7,9}}$,  

\noindent $L_{{76}}=P_{{2,4,5,6,8}}-P_{{1,3,6,7,8}}+6 P_{
{1,3,4,5,10}}+3 P_{{1,2,4,6,10}}$,  

\noindent $L_{{77}}=2 P_{{2,4,5,7,8}}+P_{{2,4,5,6,9}}+3 
P_{{1,4,6,7,8}}+6 P_{{1,4,5,7,9}}$,  

\noindent $L_{{78}}=-P_{{3,4,5,6,8}}-2 P_{{2,4,5,7,8}}-P_
{{2,4,5,6,9}}+P_{{1,3,6,7,9}}$,  

\noindent $L_{{79}}=-2 P_{{3,4,5,7,8}}+P_{{3,4,5,6,9}}+P_
{{2,4,6,7,8}}+2 P_{{2,3,4,8,9}}$,  

\noindent $L_{{80}}=2 P_{{3,4,5,7,8}}+P_{{3,4,5,6,9}}+P_{
{2,4,6,7,8}}+2 P_{{2,4,5,7,9}}$,  

\noindent $L_{{81}}=2 P_{{3,4,5,8,9}}+3 P_{{3,4,5,6,10}}-
P_{{2,4,6,8,9}}+6 P_{{2,3,4,8,10}}$,  

\noindent $L_{{82}}=P_{{3,5,6,7,10}}+6 P_{{3,4,5,8,10}}-3 
P_{{2,4,6,8,10}}+2 P_{{2,3,7,8,10}}$,  

\noindent $L_{{83}}=P_{{3,5,6,8,10}}-P_{{2,5,7,8,10}}-P_{{
2,5,6,9,10}}+3 P_{{1,5,7,9,10}}$,  

\noindent $L_{{84}}=-P_{{4,5,6,7,8}}+2 P_{{3,4,5,8,9}}-P_
{{2,4,6,8,9}}+P_{{2,3,6,7,10}}$,  

\noindent $L_{{85}}=-P_{{4,5,6,7,9}}+P_{{3,4,6,8,9}}-P_{{2
,4,7,8,9}}+3 P_{{2,3,4,9,10}}$,  

\noindent $L_{{86}}=P_{{4,5,6,8,9}}+6 P_{{3,4,5,8,10}}-3 
P_{{2,4,6,8,10}}+P_{{2,3,6,9,10}}$,  

\noindent $L_{{87}}=P_{{4,5,6,8,9}}-P_{{3,5,6,7,10}}-6 P_
{{1,4,7,8,10}}+3 P_{{1,4,6,9,10}}$,  

\noindent $L_{{88}}=3 P_{{4,5,6,8,10}}-P_{{3,5,7,8,10}}-P
_{{3,5,6,9,10}}+P_{{2,5,7,9,10}}$,  

\noindent $L_{{89}}=6 P_{{4,5,6,8,10}}-2 P_{{3,5,7,8,10}}-
P_{{3,5,6,9,10}}+3 P_{{1,6,7,9,10}}$,  

\noindent $L_{{90}}=2 P_{{4,5,7,8,9}}-3 P_{{4,5,6,7,10}}-6 
P_{{2,4,7,8,10}}+3 P_{{2,4,6,9,10}}$,  

\noindent $L_{{91}}=3 P_{{4,5,7,8,10}}-3 P_{{4,5,6,9,10}}-
P_{{3,6,7,8,10}}+P_{{3,5,7,9,10}}$,  

\noindent $L_{{92}}=6 P_{{4,5,7,8,10}}-3 P_{{4,5,6,9,10}}-2 
P_{{3,6,7,8,10}}+P_{{2,6,7,9,10}}$,  

\noindent $L_{{93}}=6 P_{{4,5,7,8,10}}-3 P_{{4,5,6,9,10}}-
P_{{3,6,7,8,10}}+3 P_{{2,4,8,9,10}}$,  

\noindent $L_{{94}}=P_{{1,5,6,7,9}}-6 P_{{1,4,5,8,9}}+9 
P_{{1,4,5,6,10}}-2 P_{{1,3,6,8,9}}+6 P_{{1,3,5,7,10}}$,  

\noindent $L_{{95}}=-P_{{2,3,5,6,7}}+2 P_{{2,3,4,5,8}}-3 
P_{{1,4,5,6,7}}+6 P_{{1,3,4,6,8}}+6 P_{{1,2,4,7,8}}$,  

\noindent $L_{{96}}=-4 P_{{2,4,5,6,8}}-2 P_{{2,3,5,7,8}}+3 
P_{{1,4,5,6,9}}+3 P_{{1,3,6,7,8}}+6 P_{{1,2,4,8,9}}$,  

\noindent $L_{{97}}=2 P_{{2,4,5,8,9}}+6 P_{{2,4,5,6,10}}-
P_{{2,3,6,8,9}}-3 P_{{1,4,6,8,9}}+18 P_{{1,3,4,8,10}}$,  

\noindent $L_{{98}}=-P_{{2,5,6,7,8}}+2 P_{{2,3,5,8,9}}+9 
P_{{1,4,5,6,10}}-3 P_{{1,3,6,8,9}}+18 P_{{1,2,4,8,10}}$,  

\noindent $L_{{99}}=-P_{{2,5,6,7,8}}+2 P_{{2,3,5,8,9}}+3 
P_{{1,5,6,7,9}}-3 P_{{1,3,6,8,9}}+6 P_{{1,2,7,8,9}}$,  

\noindent $L_{{100}}=P_{{2,5,6,7,9}}-2 P_{{2,4,5,8,9}}+3 
P_{{2,4,5,6,10}}-6 P_{{1,4,6,8,9}}+18 P_{{1,4,5,7,10}}$,  

\noindent $L_{{101}}=2 P_{{2,5,6,7,10}}+6 P_{{2,4,5,8,10}}-
P_{{2,3,6,8,10}}-9 P_{{1,4,6,8,10}}+6 P_{{1,3,7,8,10}}$,  

\noindent $L_{{102}}=3 P_{{3,4,5,6,8}}+6 P_{{2,4,5,7,8}}+3 
P_{{2,4,5,6,9}}+P_{{2,3,6,7,8}}+2 P_{{2,3,5,7,9}}$,

\noindent $L_{{103}}=-2 P_{{3,5,6,7,8}}+5 P_{{2,4,5,8,9}}-3 
P_{{2,4,5,6,10}}-3 P_{{1,4,6,8,9}}+9 P_{{1,2,4,9,10}}$,

\noindent $L_{{104}}=-P_{{3,5,6,8,9}}+2 P_{{2,5,7,8,9}}+3 
P_{{2,5,6,7,10}}-9 P_{{1,4,6,8,10}}+18 P_{{1,4,5,9,10
}}$,

\noindent $L_{{105}}=-P_{{3,5,6,8,9}}+2 P_{{2,5,7,8,9}}+3 
P_{{2,5,6,7,10}}-3 P_{{2,3,6,8,10}}+6 P_{{2,3,5,9,10}
}$,

\noindent $L_{{106}}=-3 P_{{4,5,6,7,8}}+P_{{3,5,6,7,9}}+6 
P_{{3,4,5,8,9}}-3 P_{{2,4,6,8,9}}+2 P_{{2,3,7,8,9}}$,

\noindent $L_{{107}}=-4 P_{{4,5,6,8,9}}+2 P_{{3,5,7,8,9}}+3 
P_{{3,5,6,7,10}}-3 P_{{2,4,6,8,10}}+6 P_{{2,4,5,9,10}
}$,

\noindent $L_{{108}}=-2 P_{{4,5,7,8,9}}+6 P_{{4,5,6,7,10}}+
P_{{3,6,7,8,9}}-3 P_{{3,4,6,8,10}}+6 P_{{3,4,5,9,10}}$,

\noindent $L_{{109}}=P_{{2,3,5,6,7}}-4 P_{{2,3,4,5,8}}+3 
P_{{1,4,5,6,7}}-6 P_{{1,3,4,6,8}}-6 P_{{1,3,4,5,9}}+2
 P_{{1,2,3,7,9}}$,

\noindent $L_{{110}}=-4 P_{{2,4,5,6,8}}-P_{{2,3,5,7,8}}-3 
P_{{1,4,5,7,8}}+2 P_{{1,3,6,7,8}}-9 P_{{1,3,4,5,10}}+
3 P_{{1,2,3,7,10}}$,

\noindent $L_{{111}}=-5 P_{{2,5,6,7,8}}+4 P_{{2,3,5,8,9}}+18
 P_{{1,4,5,8,9}}-27 P_{{1,4,5,6,10}}-3 P_{{1,3,6,8,9
}}+18 P_{{1,2,3,9,10}}$,

\noindent $L_{{112}}=P_{{2,5,6,8,9}}+2 P_{{2,3,5,8,10}}-2 
P_{{1,5,7,8,9}}-3 P_{{1,5,6,7,10}}-3 P_{{1,3,6,8,10}}
+3 P_{{1,2,6,9,10}}$,

\noindent $L_{{113}}=-5 P_{{3,4,5,6,8}}-8 P_{{2,4,5,7,8}}+
P_{{2,3,6,7,8}}-6 P_{{2,3,4,5,10}}+3 P_{{1,4,6,7,8}}+
18 P_{{1,2,4,7,10}}$,

\noindent $L_{{114}}=-P_{{3,4,5,6,8}}-2 P_{{2,4,5,7,8}}+2 
P_{{2,4,5,6,9}}+P_{{2,3,6,7,8}}+3 P_{{1,4,6,7,8}}+6 
P_{{1,3,4,8,9}}$,

\noindent $L_{{115}}=2 P_{{3,4,5,6,8}}+2 P_{{2,4,5,7,8}}-
P_{{2,3,6,7,8}}+6 P_{{2,3,4,5,10}}-3 P_{{1,4,6,7,8}}+
9 P_{{1,3,4,6,10}}$,

\noindent $L_{{116}}=-2 P_{{3,5,6,7,8}}+2 P_{{2,5,6,7,9}}+6 
P_{{2,4,5,8,9}}-P_{{2,3,6,8,9}}-9 P_{{1,4,6,8,9}}+6 
P_{{1,3,7,8,9}}$,

\noindent $L_{{117}}=2 P_{{3,5,6,7,8}}+P_{{2,5,6,7,9}}-6 
P_{{2,4,5,8,9}}+9 P_{{2,4,5,6,10}}-2 P_{{2,3,6,8,9}}+
6 P_{{2,3,5,7,10}}$,

\noindent $L_{{118}}=5 P_{{3,5,6,8,9}}-4 P_{{2,5,7,8,9}}-3 
P_{{2,5,6,7,10}}+18 P_{{2,4,5,8,10}}-27 P_{{1,4,6,8,
10}}+18 P_{{1,2,7,9,10}}$,

\noindent $L_{{119}}=-5 P_{{4,5,6,7,8}}-P_{{3,5,6,7,9}}+8 
P_{{3,4,5,8,9}}-3 P_{{3,4,5,6,10}}-6 P_{{1,4,7,8,9}}+
18 P_{{1,3,4,9,10}}$,

\noindent $L_{{120}}=-2 P_{{4,5,6,7,8}}-P_{{3,5,6,7,9}}+2 
P_{{3,4,5,8,9}}-3 P_{{3,4,5,6,10}}-6 P_{{1,4,7,8,9}}+
9 P_{{1,4,6,7,10}}$,

\noindent $L_{{121}}=P_{{4,5,6,7,8}}+P_{{3,5,6,7,9}}-2 P_
{{3,4,5,8,9}}+3 P_{{3,4,5,6,10}}-2 P_{{2,4,6,8,9}}+6 
P_{{2,4,5,7,10}}$,

\noindent $L_{{122}}=4 P_{{4,5,6,8,9}}-P_{{3,5,7,8,9}}-2 
P_{{3,5,6,7,10}}+3 P_{{3,4,5,8,10}}-9 P_{{1,4,7,8,10}
}+3 P_{{1,3,7,9,10}}$,

\noindent $L_{{123}}=4 P_{{4,5,7,8,9}}-6 P_{{4,5,6,7,10}}-
P_{{3,6,7,8,9}}+3 P_{{3,4,6,8,10}}-6 P_{{2,4,7,8,10}}
+2 P_{{2,3,7,9,10}}$,

\noindent $L_{{124}}=-4 P_{{2,5,6,7,8}}+2 P_{{2,3,5,8,9}}-3 
P_{{1,5,6,7,9}}+18 P_{{1,4,5,8,9}}-27 P_{{1,4,5,6,10}
}-3 P_{{1,3,6,8,9}}+9 P_{{1,2,6,7,10}}$,

\noindent $L_{{125}}=-5 P_{{3,5,6,7,8}}-P_{{2,5,6,7,9}}+12 
P_{{2,4,5,8,9}}-9 P_{{2,4,5,6,10}}-P_{{2,3,6,8,9}}-9 
P_{{1,4,6,8,9}}+9 P_{{1,3,6,7,10}}$,

\noindent $L_{{126}}=4P_{{3,5,6,8,9}}-2P_{{2,5,7,8,9}}-3
P_{{2,5,6,7,10}}+18P_{{2,4,5,8,10}}-3P_{{2,3,6,8,10
}}-27P_{{1,4,6,8,10}}+9P_{{1,3,6,9,10}}$.

\end{document}